\documentclass[a4paper,12pt]{amsart}
\usepackage{fullpage, verbatim}
\usepackage{enumerate, amsfonts, amssymb, amsthm, mathrsfs}
\usepackage{xcolor}

\usepackage[all]{xy}
\usepackage[T1]{fontenc}

\usepackage[colorlinks=true]{hyperref}
\numberwithin{equation}{section}
\usepackage{cleveref}

\newcommand{\loccit}{{\it{loc.~cit.}}}

\newcommand\CA{{\mathscr A}}

\newcommand\FD{\mathfrak D}

\newcommand\BBC{{\mathbb C}}

\newcommand\BBN{{\mathbb N}}
\newcommand\BBQ{{\mathbb Q}}

\newcommand\BBZ{{\mathbb Z}}

\newcommand\Yprime{Y}

\DeclareMathOperator{\spec}{spec}

\DeclareMathOperator{\diag}{diag}

\newcommand\ord{{\operatorname{ord}}}

\newcommand\Hodgefil{\mathcal{H}}
\newcommand\Hodgedec{\mathcal{G}}

\newcommand\Binf{B_\infty}

\newcommand\weight{\mathrm w}

\newcommand{\Mprime}{M^{(0)}}

\newcommand\coexp{\operatorname{coexp}}

\newcommand\Coinv{\operatorname{Coinv}}

\newcommand\Der{{\operatorname{Der}}}
\newcommand\DerS{{\Der_S}}
\newcommand\DerSW{{\Der_S^W}}

\newcommand\DerR{{\Der_R}}

\newcommand\InvSpecial{F^{\mathrm{fl}}}
\newcommand\hh{\Delta}
\newcommand\VanLoc{\mathcal{H}}

\newcommand\Fix{{\operatorname{Fix}}}

\newcommand\GL{\operatorname{GL}}

\newcommand\Irr{{\operatorname{Irr}}}

\newcommand\pdeg{\operatorname{pdeg}}

\newcommand\refl{\mathcal{R}}

\newcommand{\prim}{\nabla_{\! D}}
\newcommand{\priminv}{\prim^{-1}}
\newcommand{\priminvm}{\prim^{-m}}

\newcommand{\eqdef}{:=}

\newcommand{\pmat}[1]{\begin{pmatrix} #1 \end{pmatrix}}

\newcommand{\colvect}{\pmat{\partial_{t_1} \\ \vdots \\ \partial_{t_\ell}}}
\newcommand{\colvecx}{\pmat{\partial_{x_1} \\ \vdots \\ \partial_{x_\ell}}}
\newcommand{\colvecXi}{\pmat{\xi_\ell \\ \vdots \\ \xi_1}}

\newcommand{\rowvect}{(\partial_{t_1}, \ldots, \partial_{t_\ell})}
\newcommand{\rowvecx}{(\partial_{x_1}, \ldots, \partial_{x_\ell})}
\newcommand{\rowvecXi}{(\xi_\ell, \ldots, \xi_1)}

\newcommand{\Jtx}{J_{\partial\bbt/\partial\bbx}}
\newcommand{\Jxt}{J_{\partial\bbx/\partial\bbt}}

\newcommand\inverse{^{-1}}

\newcommand{\bbt}{{\textbf{t}}}

\newcommand{\bbx}{{\textbf{x}}}

\newcommand{\tr}{\mathrm{tr}}

\newcommand\id{{id}}

\definecolor{darkblue}{rgb}{0,0,0.7} 
\newcommand{\darkblue}{\color{darkblue}} 
\newcommand{\defn}[1]{\emph{\darkblue #1}} 

\newcommand{\ie}{{i.e.}} 
\newcommand{\eg}{{e.g.}} 
\newcommand{\cf}{\text{cf.}} 

\newcommand{\one}{{1\!\!1}} 

\usepackage[colorinlistoftodos]{todonotes}

\theoremstyle{plain}
\newtheorem{lemma}[equation]{Lemma}
\newtheorem{theorem}[equation]{Theorem}

\newtheorem{corollary}[equation]{Corollary}
\newtheorem{proposition}[equation]{Proposition}
\theoremstyle{definition}
\newtheorem{definition}[equation]{Definition}
\newtheorem{remark}[equation]{Remark}
\newtheorem{remarks}[equation]{Remarks}


\begin{document}

\title[Freeness of multi-reflection arrangements]
{Freeness of multi-reflection arrangements\\ via primitive vector fields}

\author[T.~Hoge]{Torsten Hoge}
\address
{Fakult\"at f\"ur Mathematik und Physik,
Leibniz Universit\"at Hannover,
Germany}
\email{hoge@math.uni-hannover.de}

\author[T.~Mano]{Toshiyuki Mano}
\address
{Department of Mathematical Sciences,
University of the Ryukyus,
Okinawa, Japan}
\email{tmano@math.u-ryukyu.ac.jp}

\author[G.~R\"ohrle]{Gerhard R\"ohrle}
\address
{Fakult\"at f\"ur Mathematik,
Ruhr-Universit\"at Bochum, Germany
}
\email{gerhard.roehrle@rub.de}

\author[C.~Stump]{Christian Stump}
\address
{Fakult\"at f\"ur Mathematik,
Ruhr-Universit\"at Bochum, Germany}
\email{christian.stump@rub.de}


\allowdisplaybreaks

\begin{abstract}
  In 2002, Terao showed that every reflection multi-arrangement of a real reflection group with constant multiplicity is free by providing a basis of the module of derivations.
  We first generalize Terao's result to 
  multi-arrangements stemming from well-generated unitary reflection groups, where the multiplicity of a hyperplane depends on the order of its stabilizer.
  Here the exponents depend on the exponents of the dual reflection representation.
  We then extend our results further to all imprimitive irreducible unitary reflection groups. In this case the exponents turn out to depend on the exponents of a certain Galois twist of the dual reflection representation that comes from a Beynon-Lusztig type semi-palindromicity of the fake degrees.
\end{abstract}

\maketitle

\setcounter{tocdepth}{1}
\tableofcontents


\section{Introduction}

In his seminal work~\cite{ziegler:multiarrangements}, Ziegler
introduced the concept of multi-arrangements
generalizing the notion of hyperplane arrangements.
In~\cite{terao:multi}, Terao showed that every reflection multi-arrangement of a real reflection group with constant multiplicities is free, see also the approach by Yoshinaga~\cite{yoshinaga:multicoxeter}.
Our aim is to generalize this result from real reflection groups to unitary reflection groups, see \Cref{thm:main1,,thm:main2}.
It turns out that the constant multiplicity in the real case stems from the fact that real reflections have order two. In the general case this constant multiplicity is replaced by the \emph{order multiplicity} given by the orders of the unitary reflections involved.

\medskip

We first extend Yoshinaga's construction of a basis of the module of derivations and of Saito's Hodge filtration to well-generated unitary reflection groups by constructing and then using a flat connection based on recent developments of flat systems of invariants in the context of isomonodromic deformations and differential equations of Okubo type due to Kato, Mano and Sekiguchi~\cite{KMS2015}, see also the recent work of Konishi, Minabe and Shiraishi~\cite{KMS2018b}.
In the case of well-generated unitary reflection groups this flat connection replaces the role of Saito's flat connection in the case of real reflection groups~\cite{Sai1993}, and naturally explains the occurrence of the order multiplicity in the multiplicity function.
The freeness in this case is thus the algebraic incarnation of the geometry of the logarithmic vector fields along the discriminant hypersurface.

\medskip

We then further extend the results to the imprimitive reflection groups by use of a permutation of the irreducible complex representations that is studied in the context of the representation theory of the Hecke algebra and which induces a semi-palindromic property on the fake degree polynomial~\cite{malle:fake,opdam:dunkoperators,gordongriffeth}.

\bigskip

Suppose that~$W$ is an irreducible unitary reflection group with
reflection representation~$V \cong \BBC^\ell$.
Denote the set of reflections by $\refl = \refl(W)$, and the associated reflection arrangement in~$V$ by $\CA = \CA(W)$.
For $H \in \CA$, let $e_H \in \BBN \eqdef \{ 0,1,2,\ldots\}$ denote the order of the pointwise
stabilizer of~$H$ in~$W$, and consider the \defn{order multiplicity} given by
\[
\omega : \CA \to\BBN, \quad \omega(H) = e_H 
\]
for $H \in \CA$.
For $m \in \BBN$ let  $m \omega$ and $m \omega + 1$ denote the multiplicities $m \omega(H) = m e_H$ and $m \omega(H) + 1 = m e_H + 1$ for $H \in \CA$, respectively.
Observe that in the case that~$W$ is real, we have $e_H = 2$ for all $H \in \CA$.
In this case, $m\omega$ thus corresponds to the constant even multiplicity $2m$ and $m\omega+1$ to the constant odd multiplicity $2m+1$.

\medskip

Following~\cite{gordongriffeth}, the \defn{Coxeter number} of~$W$ is given by
\[
h = h_W \eqdef \frac{1}{\ell} \sum_{H \in \CA} e_H = \frac{1}{\ell}\Big(|\refl| + |\CA|\Big),
\]
generalizing the usual Coxeter number of a real reflection group to irreducible unitary reflection groups.
Let $\Irr(W)$ denote the irreducible complex representations of~$W$ up to isomorphism.
For~$U$ in $\Irr(W)$ of dimension~$d$,
denote by
\[
  \exp_U(W) \eqdef \big\{n_1(U) \leq \ldots \leq n_{d}(U) \big\}
\]
the \defn{$U$-exponents} of~$W$ given by the~$d$ homogeneous degrees in the coinvariant algebra of~$W$ in which~$U$ appears.
In particular, the \defn{exponents} of~$W$ are
\[
  \exp(W) \eqdef  \exp_V(W) = \big\{n_1(V) \leq \ldots \leq n_{\ell}(V)\big\}
\]
and the \defn{coexponents} of~$W$ are 
\[
  \coexp(W) \eqdef  \exp_{V^*}(W) = \big\{n_1(V^*) \leq \ldots \leq n_{\ell}(V^*)\big\}.
\]
The group~$W$ is \defn{well-generated} if $n_i(V) + n_{\ell+1-i}(V^*) = h$, \eg, see~\cite{OS1980,malle:fake,Bes2015}.

\medskip

Our first main result generalizes Terao's theorem~\cite{terao:multi} to the well-generated case.
 
\begin{theorem}
\label{thm:main1}
  Let~$W$ be an irreducible, well-generated unitary reflection group with reflection arrangement $\CA(W)$.
  Let $\omega : \CA(W) \to \BBN$ given by $\omega(H) = e_H$, and let $m \in \BBN$.
  Then
  \begin{enumerate}[(i)]
    \item\label{eq:main11}
      the reflection multi-arrangement $(\CA(W), m \omega)$ is free 
      with exponents
      \[
        \exp (\CA(W), m \omega) = \big\{m h, \ldots, m h \big\},
      \]
    \item\label{eq:main12}
      the reflection multi-arrangement $(\CA(W), m \omega + 1)$ is free 
      with exponents
      \begin{equation*}
        \exp (\CA(W), m\omega+1) = \big\{m h + n_1(V^*), \ldots , m h + n_\ell(V^*) \big\}.
      \end{equation*}
  \end{enumerate}
\end{theorem}

Note from above that $\coexp(W) = \exp_{V^*}(W) = \big\{n_1(V^*), \ldots, n_{\ell}(V^*)\big\}$.

\medskip

In the special case when $W$ is a Coxeter group, \Cref{thm:main1} recovers Terao's theorem ~\cite{terao:multi},
as then $\omega \equiv 2$ and $\coexp(W) = \exp(W)$.

We prove this theorem in \Cref{sec:proofmain1}.
Indeed, we extend Yoshinaga's construction~\cite[Thm.~1]{yoshinaga:multicoxeter} of a basis of the module of derivations and Saito's Hodge filtration to well-gen\-erated groups by using a recent construction due to Kato, Mano and Sekiguchi~\cite{KMS2015}.
See \Cref{thm:main1strong} for the precise formulation, which is our generalization of~\cite[Thm.~7]{yoshinaga:multicoxeter} to the well-generated setting.

\medskip

In~\cite{KMS2015}, the authors construct flat systems of invariants of well-generated unitary reflection groups in the context of isomonodromic deformations and differential equations of Okubo type.
For real reflection groups, the notion of flat systems of invariants was introduced by Saito, Yano and Sekiguchi in~\cite{SYS1980}.
The existence of such flat systems was shown in \loccit\ in all real types except~$E_7$ and~$E_8$.
Saito then gave a uniform construction in all real types in~\cite{Sai1993}.

\bigskip

Our second main result extends \Cref{thm:main1} further to the infinite three-parameter family $W = G(r,p,\ell)$ of imprimitive reflection groups.
It turns out that the corresponding multi-arrangements are also free.
However, the description of the exponents is considerably more involved and depends on the representation theory of the Hecke algebra associated to the group~$W$.
To this end, let~$\Psi$ denote the permutation on $\Irr(W)$ introduced by Malle in~\cite[Sec.~6C]{malle:fake}, having the semi-palindromic property on the fake degrees of~$W$.
This is, for any~$U$ in $\Irr(W)$ of dimension~$d$, we have
\begin{equation*}
  n_i(U) + n_{d+1-i}(\Psi(U^*)) = h_U,
\end{equation*}
where $h_U = |\CA|- \sum_{r \in \refl}\chi(r)/\chi(1)$ with $\chi$ being the character of~$U$.
A direct calculation shows that $h_V = h_W$ is the Coxeter number of~$W$.
Moreover, $\Psi(V^*) = V^*$ if and only if~$W$ is well-generated~\cite[Cor.~4.9]{malle:fake}.

\medskip

\begin{theorem}
\label{thm:main2}
  Let~$W = G(r,p,\ell)$ with reflection arrangement $\CA(W)$.
  Let $\omega : \CA(W) \to \BBN$ given by $\omega(H) = e_H$, and let $m \in \BBN$.
  Then
  \begin{enumerate}[(i)]
  \item\label{eq:main21}
    the reflection multi-arrangement $(\CA(W), m \omega)$ is free 
    with exponents
    \[
      \exp (\CA(W), m \omega) = \big\{m h, \ldots, m h \big\},
    \]
  \item\label{eq:main22}
    the reflection multi-arrangement $(\CA(W), m \omega + 1)$ is free 
    with exponents
    \[
      \exp (\CA(W),  m \omega+1) = 
      \big\{m h +n_1(\Psi^{-m}(V^*)), \ldots , m h + n_\ell(\Psi^{-m}(V^*)) \big\}.
    \]
  \end{enumerate}
\end{theorem}

Note this time that $\exp_{\Psi^{-m}(V^*)}(W) = \big\{n_1(\Psi^{-m}(V^*)), \ldots , n_\ell(\Psi^{-m}(V^*)) \big\}$.
We prove a more general result in \Cref{thm:main2strong}.

\begin{remarks}
  (i) The group $G(r,p,\ell)$ is well-generated if and only if $p \in \{1,r\}$.
       Thus, \Cref{thm:main2} extends \Cref{thm:main1} to the class of imprimitive reflection groups that are not well-generated.

  (ii) While the simple arrangements of the reflection groups $G(r,1,\ell)$ and $G(r,p,\ell)$ for $1 < p < r$ coincide, the multi-arrangements above depend on the underlying group, since the multiplicities of the coordinate hyperplanes differ.

  (iii) \Cref{thm:main1,,thm:main2} only leave unresolved the remaining eight irreducible unitary reflection groups of exceptional type that are not well-generated, namely
  \[
    \mathcal{G}_{\operatorname{exc}} = \big\{ G_7,G_{11},G_{12},G_{13},G_{15},G_{19},G_{22},G_{31} \big\}.
  \]
  Computational evidence for each of these remaining groups with small values for the parameter~$m \in \BBN$ suggests that \Cref{thm:main2} also holds with $W = G(r,p,\ell)$ replaced by $W \in \mathcal{G}_{\operatorname{exc}}$.
  Note that all these groups are of rank~$2$ except for $G_{31}$ and that multi-arrangements of rank~$2$ are always 
  free~\cite[Cor.~7]{ziegler:multiarrangements}.

  (iv) The semi-palindromic property of the permutation~$\Psi$ of~$\Irr(W)$ in \Cref{thm:main2} is an analogue of a semi-palindromicity of the fake degrees as observed by Beynon and Lusztig~\cite[Prop.~A]{BL1978} and later explained by Opdam~\cite{Opd1995}.
        The definition of~$\Psi$ depends on the representation theory of the corresponding Hecke algebra~\cite{malle:fake,opdam:dunkoperators}.
        Moreover, it plays a crucial role in the study of rational Cherednik algebras~\cite[Thm.~1.6]{gordongriffeth}.
        The intrinsic appearance of~$\Psi$ in the present context of multi-derivations of reflection groups is rather unexpected.
\end{remarks}

\bigskip

The paper is organized as follows.
In \Cref{sect:preliminaries}, we provide all needed background on hyperplane arrangements and unitary reflection groups.
The proof of \Cref{thm:main1} is carried out in \Cref{sec:proofmain1}, along with its strengthened form, \Cref{thm:main1strong}.
\Cref{thm:main2} is proved in the final \Cref{sec:proofmain2} as a consequence of \Cref{thm:main2strong}.

\section{Preliminaries}
\label{sect:preliminaries}

We first provide some basic material on hyperplane arrange\-ments and multi-arrange\-ments, and their modules of derivations.
We then recall the needed background on unitary reflection groups.
For general information about reflection groups and their arrangements, we refer the reader to~\cite{bourbaki:groupes,orliksolomon:unitaryreflectiongroups,ziegler:multiarrangements,orlikterao:arrangements}.

\subsection{Multi-arrangements and their modules of derivations}
\label{ssect:free}

Let $S = S(V^*)$ denote the \defn{ring of polynomial functions} on~$V$ considered as the symmetric algebra of the dual space~$V^*$.
If $x_1, \ldots, x_\ell$ is a basis of~$V^*$, we identify~$S$ with the polynomial ring $\BBC[x_1, \ldots , x_\ell]$.
Letting $S_p$ denote the $\BBC$-subspace of~$S$ consisting of the homogeneous polynomials of degree $p$ (along with $0$),~$S$ is naturally $\BBZ$-graded by $S = \oplus_{p \in \BBZ}S_p$, where we consider $S_p = 0$ for $p < 0$.

Let $\DerS$ be the \defn{$S$-module of $\BBC$-derivations} of~$S$.
Then $\partial_{x_1}, \ldots, \partial_{x_\ell}$ is an $S$-basis of $\DerS$.
We say that $\theta \in \DerS$ is \defn{homogeneous of polynomial degree $p$} provided $\theta = \sum f_i \partial_{x_i}$, where $f_i \in S_p$ for each $1 \le i \le \ell$.
In this case we write $\pdeg \theta = p$.       
Let $\Der_{S_p}$ be the $\BBC$-subspace of $\DerS$ consisting of all homogeneous derivations of polynomial degree~$p$.
Then $\DerS$ is a graded~$S$-module, $\DerS = \oplus_{p\in \BBZ} \Der_{S_p}$.

\medskip

A \defn{hyperplane arrangement} $\CA$ in~$V$ is a finite collection of hyperplanes in~$V$.
For a subspace~$X$ of~$V$, we have the associated \defn{localization} of~$\CA$ at~$X$ given by
\[
  \CA_X \eqdef\{H \in \CA \mid X \subseteq H\} \subseteq \CA.
\]
Its \defn{rank} is defined to be the codimension of~$X$ in~$V$.

\medskip

Following Ziegler~\cite{ziegler:multiarrangements}, a \defn{multi-arrangement} $(\CA,\nu)$ is an arrangement~$\CA$ together with a \defn{multiplicity function} $\nu : \CA \to \BBN$ assigning to each hyperplane~$H\in\CA$ a multiplicity $\nu(H) \in \BBN$.
If $\nu \equiv 1$, then $(\CA,\nu)$ is called \defn{simple}.
We only consider \emph{central} multi-arrangements $(\CA,\nu)$, \ie, $0 \in H$ for every $H \in \CA$.
In this case, we fix $\alpha_H \in V^*$ with $H = \ker(\alpha_H)$ for $H \in \CA$.
The \defn{order} of $(\CA, \nu)$ is given by
\[
  |\nu| \eqdef |(\CA, \nu)| \eqdef \sum_{H \in \CA} \nu(H),
\]
and its \defn{defining polynomial} $Q(\CA,\nu) \in S$ is
\[
  Q(\CA,\nu) \eqdef \prod_{H \in \CA} \alpha_H^{\nu(H)}.
\]

\medskip

The \defn{module of derivations} of $(\CA, \nu)$ is defined by 
\[
  \FD(\CA, \nu) \eqdef \big\{\theta \in \DerS \mid \theta(\alpha_H) \in S \alpha_H^{\nu(H)} \text{ for each } H \in \CA \big\}.
\]

We say that $(\CA, \nu)$ is \defn{free} if $\FD(\CA, \nu)$ is a free~$S$-module~\cite[Def.~6]{ziegler:multiarrangements}.
In this case, $\FD(\CA, \nu)$ admits a basis $\{\theta_1, \ldots, \theta_\ell\}$ of~$\ell$ homogeneous derivations~\cite[Thm.~8]{ziegler:multiarrangements}.
While the $\theta_i$'s are not unique, their polynomial degrees $\pdeg \theta_i$ are.
The multiset of these polynomial degrees is the set of \defn{exponents} of the free multi-arrangement $(\CA,\nu)$.
It is denoted by
\[
  \exp(\CA,\nu) \eqdef \big\{ \pdeg(\theta_1),\ldots, \pdeg(\theta_\ell) \big\}.
\]

Next we record Ziegler's analogue of Saito's criterion.
The \defn{Saito matrix} of $\theta_1, \ldots, \theta_\ell \in \DerS$ is given by
\[
  M(\theta_1, \ldots, \theta_\ell) \eqdef
  \begin{bmatrix}
    \theta_1(x_1) & \cdots & \theta_1(x_\ell) \\
    \vdots & \ddots & \vdots \\
    \theta_\ell(x_1) & \cdots & \theta_\ell(x_\ell)
  \end{bmatrix},
\]
see~\cite[Def.~4.11]{orlikterao:arrangements}.

\begin{theorem}[{\cite[Thm.~8]{ziegler:multiarrangements}}]
\label{thm:zieglersaito}
  Let $(\CA, \nu)$ be a multi-arrangement, and let $\theta_1, \ldots, \theta_\ell \in \FD(\CA, \nu)$.
  Then the following are equivalent:
  \begin{enumerate}[(i)]
  \item $\{\theta_1, \ldots, \theta_\ell\}$ is an ~$S$-basis of $\FD(\CA, \nu)$.
  \item $\det M(\theta_1, \ldots, \theta_\ell) \doteq  Q(\CA, \nu)$.
  \end{enumerate}
In particular, if each $\theta_i$ is homogeneous, then both are moreover equivalent 
to the following:
\begin{enumerate}[(i)]
  \setcounter{enumi}{2}
  \item\label{eq:zieglersaito3} $\theta_1, \ldots, \theta_\ell$ are linearly independent over~$S$ and $\sum \pdeg \theta_i = \deg Q(\CA, \nu) = |\nu|$. 
\end{enumerate}
\end{theorem}

In the statement and later on, the sign~$\doteq$ denotes, as usual, equality up to a non-zero complex constant.
Terao's celebrated \emph{Addition-Deletion Theorem}~\cite{terao:freeI} plays a crucial role in the study of free arrangements, see~\cite[Thm.~4.51]{orlikterao:arrangements}.
We next describe its version for multi-arrange\-ments from~\cite{abeteraowakefield:euler}.
Let $(\CA, \nu)$ be a non-empty multi-arrangement, \ie, $|\nu| \geq 1$.
Fix $H_0$ in $\CA$ with $\nu(H_0) \geq 1$.
Its \defn{deletion} with respect to $H_0$ is given by $(\CA', \nu')$, where $\nu'(H_0) = \nu(H_0) - 1$ and $\nu'(H) = \nu(H)$ for all $H \neq H_0$.
If $\nu'(H_0) = 0$, we set $\CA' = \CA\setminus\{H_0\}$, and else set $\CA' = \CA$.
Its \defn{restriction} with respect to $H_0$ is given by $(\CA'', \nu^*)$, where $\CA'' = \{ H \cap H_0 \mid H \in \CA \setminus \{H_0\}\ \}$.
The \defn{Euler multiplicity} $\nu^*$ of $\CA''$ is defined as follows.
Let $Y \in \CA''$.
Since the localization $\CA_Y$ is of rank $2$, the multi-arrangement $(\CA_Y, \nu_Y)$ is free where we set $\nu_Y = \nu|_Y$ to be the restriction of~$\nu$ to $\CA_Y$ \cite[Cor.~7]{ziegler:multiarrangements}.
According to \cite[Prop.~2.1]{abeteraowakefield:euler}, the module of derivations $\FD(\CA_Y, \nu_Y)$ admits a particular homogeneous basis $\{\theta_Y, \psi_Y, \partial_3, \ldots, \partial_\ell\}$, where $\theta_Y$ is identified by the property that $\theta_Y \notin \alpha_0 \DerS$ and $\psi_Y$ by the property that $\psi_Y \in \alpha_0 \DerS$, where $H_0 = \ker \alpha_0$.
Then the Euler multiplicity $\nu^*$ is defined on $Y$ as $\nu^*(Y) = \pdeg \theta_Y$.
Crucial for our purpose is the fact that the value $\nu^*(Y)$ only depends on the~$S$-module $\FD(\CA_Y, \nu_Y)$.
Sometimes, $(\CA, \nu), (\CA', \nu')$ and $(\CA'', \nu^*)$ is referred to as the \emph{triple} of multi-arrangements with respect to $H_0$.

\begin{theorem}[{\cite[Thm.~0.8]{abeteraowakefield:euler}}]
\label{thm:add-del-multi}
  Suppose that $(\CA, \nu)$ is not empty, fix $H_0$ in $\CA$ and let  $(\CA, \nu), (\CA', \nu')$ and  $(\CA'', \nu^*)$ be the triple with respect to $H_0$. 
  Then any  two of the following statements imply the third:
  \begin{enumerate}[(i)]
    \item $(\CA, \nu)$ is free with $\exp (\CA, \nu) = \{ b_1, \ldots , b_{\ell -1}, b_\ell\}$;
    \item $(\CA', \nu')$ is free with $\exp (\CA', \nu') = \{ b_1, \ldots , b_{\ell -1}, b_\ell-1\}$;
    \item $(\CA'', \nu^*)$ is free with $\exp (\CA'', \nu^*) = \{ b_1, \ldots , b_{\ell -1}\}$.
  \end{enumerate}
\end{theorem}

We need the following fact in the sequel.

\begin{lemma}[{\cite[Prop.~4.1(1)]{abeteraowakefield:euler}}]
  \label{prop:Euler}
  Let $H_0\in\CA$.
  Suppose $X\in\CA^{H_0}$ with $\CA_X = \{ H_0, H\}$.
  For $\nu$ a multiplicity on $\CA$, we have
  $
    \nu^*(X) = \nu(H).
  $
\end{lemma}

\subsection{Unitary Reflection Groups}
\label{subsec:unitaryreflectiongroups}

Let $V \cong \BBC^\ell$, and consider a finite subgroup~$W$ of $\GL(V)$.
Then $W$ is a \defn{unitary reflection group} if it is generated by its subset~$\refl = \refl(W)$ of \defn{reflections}, that is, the elements~$r \in W$ for which the \defn{fixed space}
\[
  \Fix(r) \eqdef \ker(\one - r) = \{ v \in V \mid r v = v \} \subseteq V
\]
is a hyperplane.
We denote by $\CA = \CA(W)$ the associated \defn{reflection arrangement} given by the collection of the reflecting hyperplanes. 
For $H \in \CA$, let $W_H = \{w \in W \mid \Fix(w) \supseteq H\}$ be
the pointwise stabilizer of~$H$ in~$W$ and set $e_H = |W_H|$.
Indeed, the elements in $W_H$ except the identity are exactly the reflections $r \in \refl$ such that $\Fix(r) = H$, explaining the equality
\begin{equation}
\label{eq:R}
  |\refl| + |\CA| = \sum_{H \in \CA} e_H.
\end{equation}
Results of Shephard and Todd~\cite{shephardtodd} and of Chevalley~\cite{chevalley} distinguish unitary reflection groups as those finite subgroups of $\GL(V)$ for which the invariant subalgebra of the action on the symmetric algebra $S = S(V^*)\cong\BBC[x_1,\ldots,x_\ell] $ yields again a polynomial algebra,
\[
  S^W = S(V^*)^W \cong \BBC[f_1,\ldots,f_\ell].
\]
While the \defn{basic invariants} $f_1,\ldots,f_\ell$ are not unique, they can be chosen to be homogeneous, and then their degrees $d_1 \leq \cdots \leq d_\ell$ are uniquely determined and called the \defn{degrees} of~$W$.

\medskip

The group~$W$ is called \defn{irreducible} if it does not preserve a proper non-trivial subspace of~$V$.
It is well-known that such an irreducible reflection group can be generated either by~$\ell$ or by~$\ell+1$ reflections.
An important subclass of irreducible unitary reflection groups are those that are \defn{well-generated}, \ie, which can be generated by~$\ell$ reflections.
In particular, this subclass contains all (complexifications of) irreducible \emph{real} reflection groups and all \emph{Shephard groups} (symmetry groups of regular complex polytopes~\cite[Def.~6.119]{orlikterao:arrangements}).

\medskip

Let $S_+^W$ denote the~$W$-invariants without constant term, and let $\Coinv(W) \eqdef S / S_+^W$ be the \defn{ring of coinvariants} of~$W$.
Observe that $\Coinv(W)$ is also a graded~$W$-module, and indeed isomorphic to the regular representation of~$W$, see~\cite[\S 4.4]{lehrertaylor}.
Thus, an irreducible representation~$U$ in $\Irr(W)$ of dimension~$d$ occurs~$d$ times in $\Coinv(W)$ as a constituent.
The \defn{$U$-exponents} of~$W$ are then given by the multiset of~$d$ homogeneous degrees in the coinvariant algebra of~$W$ in which~$U$ appears,
\begin{equation*}
  \label{eq:expU} \exp_U(W) = \{n_1(U) \leq \ldots \leq n_{d}(U)\}.
\end{equation*}
In particular, $\exp(W) = \exp_V(W)$ are the \defn{exponents of~$W$} and $\coexp(W) = \exp_{V^*}(W)$ are the \defn{coexponents of~$W$}.
It is moreover well-known that the degrees of~$W$ and the exponents are related by $d_i = n_i(V) + 1$, implying
\begin{equation}
\label{eq:R2}
  |\refl| = \sum_{i=1}^\ell n_i(V).
\end{equation}

Terao showed in~\cite{terao:freereflections} that the reflection arrangement $\CA$ of~$W$ is free,
and that the exponents of the arrangement coincide with the coexponents of~$W$, \cf~\cite[Thm.~6.60]{orlikterao:arrangements},
\[
  \exp \CA = \coexp(W).
\]

Consequently, thanks to~\cite[Thm.~4.23]{orlikterao:arrangements},
 we have 
\begin{equation}
\label{eq:a}
  |\CA| = \sum_{i=1}^\ell n_i(V^*).
\end{equation}

The next definition can be found in~\cite{gordongriffeth}.
The two equalities follow from~\eqref{eq:R},~\eqref{eq:R2}, and~\eqref{eq:a}.

\begin{definition}
\label{defn:coxeternr}
  Let~$W$ be an irreducible unitary reflection group.
  The \defn{Coxeter number}~$h = h_W$ is defined as
  \[
  h \eqdef h_W \eqdef \frac{1}{\ell} \sum_{H \in \CA} e_H = \frac{1}{\ell}\big(|\refl| + |\CA|\big) = \frac{1}{\ell} \sum_{i=1}^\ell\big(n_i(V) + n_i(V^*)\big).
  \]
\end{definition}

\begin{remark}
\label{rem:well-generated}
  It was observed by Orlik and Solomon in~\cite[Thm.~5.5]{OS1980} that the group~$W$ is well-generated if and only if the exponents and the coexponents pairwise sum up to the Coxeter number.
  This is,
  \[
    n_i(V) + n_{\ell+1-i}(V^*) = h
  \]
  for all $1 \leq i \leq \ell$.
  In this case, the Coxeter number $h = d_\ell = n_\ell(V) + 1 > d_{\ell-1}$ is the unique largest degree of a fundamental invariant, see~\cite[\S 12.6]{lehrertaylor}.
\end{remark}

The \defn{fake degree} of~$U$ in $\Irr(W)$ of dimension~$d$ is defined to be the polynomial 
\[
  f_U(q) \eqdef \sum_{i=1}^{d} q^{ n_i(U) } \in \BBN[q],
\]
\cf~\cite[Eq.~(6.1)]{malle:fake}.
In~\cite[Thm.~6.5]{malle:fake}, Malle showed that there is a permutation $\Psi$ of $\Irr(W)$ 
so that the fake degree polynomials $f_U(q)$ satisfy the \emph{semi-palindromic} condition
\begin{equation}
\label{eq:palindromic}
  f_U(q) = q^{h_U} f_{\Psi(U^*)}(q\inverse),
\end{equation}
where
\begin{align}
\label{eq:hU}
  h_U \eqdef |\refl|- \sum_{r \in \refl}\chi_U(r)/\chi_U(1).
\end{align}
Equivalently,~$h_U$ is the integer by which the central element $\sum_{r \in \refl}(\one-r) \in\BBC[W]$ acts on~$U$.
In particular, for any~$U$ in $\Irr(W)$ of dimension~$d$, we have
\[
  n_i(U) + n_{d+1-i}(\Psi(U^*)) = h_U.
\]
The following observations provide, for later reference, the formula in \Cref{thm:main2}\eqref{eq:main22} in a form analogous to the one used in~\cite[Sec.~3]{gordongriffeth}.

\begin{lemma}
\label{lem:hUstability}
  The parameter~$h_U$ defined in~\eqref{eq:hU} satisfies $h_U = h_{U^*}$ and $h_U = h_{\Psi(U)}$.
  In particular, we have, for any $1 \leq i \leq\ell$ and any $m \in \BBN$,
  \begin{align}
  \label{eq:constant_h}
    m h +n_i(\Psi^{-m}(V^*)) = (m+1)h -n_{\ell +1-i}(\Psi^{-m-1}(V^*)^*). 
  \end{align}
\end{lemma}

\begin{proof}
  The equality $h_U = h_{U^*}$ is a direct consequences of~\eqref{eq:hU}.
  The equality $h_U = h_{\Psi(U)}$ follows, for example, from the description of $\Psi$ as the operator $\phi^{\id}_{-\frac{1}{h},-1-\frac{1}{h}}$ in~\cite[\S 2.12]{gordongriffeth} together with the observation in~\cite[\S 2.8]{gordongriffeth} that $h_{\phi^{\id}_{-\frac{1}{h},-1-\frac{1}{h}}(U)} = h_U$.
  Plugging in $\Psi^{-m-1}(V^*)^*$ for the irreducible representation~$U$ in~\eqref{eq:palindromic} 
  and using that $h_{\Psi^{-m-1}(V^*)^*} = h_V = h$ yields~\eqref{eq:constant_h}.
\end{proof}

See also~\cite[Prop.~7.4]{opdam:dunkoperators} and~\cite[\S~1.4]{gordongriffeth} for further properties of the permutation $\Psi$ of $\Irr(W)$.
Note that~$\Psi(V^*) = V^*$ if and only if~$W$ is well-generated~\cite[Cor.~4.9]{malle:fake}.

\medskip

We finally define the \defn{order multiplicity} $\omega$ of the reflection arrangement
$\CA = \CA(W)$ by $\omega(H) = e_H$ for $H \in \CA$.
In other words, the multiplicities are chosen so that the defining polynomial $Q(\CA(W), \omega)$ of the multi-arrangement $(\CA(W), \omega)$ is the discriminant of~$W$, \cf~\cite[Def.~6.44]{orlikterao:arrangements},
\[
  Q(\CA(W), \omega) = \prod_{H \in \CA(W)} \alpha_H^{e_H}.
\]

\section{Proof of \Cref{thm:main1}}
\label{sec:proofmain1}

In this section, we prove a strengthened version of \Cref{thm:main1}.
Our method is based on the approach by Yoshinaga~\cite{yoshinaga:multicoxeter}, also relying strongly on recent developments of flat systems of invariants for well-generated unitary reflection groups in the context of isomonodromic deformations and differential equations of Okubo type due to Kato, Mano and Sekiguchi~\cite{KMS2015}.
See \Cref{thm:main1strong} for the explicit formulation.

\bigskip

Let $\nabla : \DerS \times \DerS \rightarrow \DerS$ be an \defn{affine connection}.
Recall that $\nabla$ is $S$-linear in the first parameter and $\BBC$-linear in the second, satisfying the Leibniz rule
\[
  \nabla_\delta(p \delta') = \delta(p)\delta' + p \nabla_\delta(\delta')
\]
for $\delta,\delta' \in \DerS$.
The connection~$\nabla$ is \defn{flat} if $\nabla_\delta(\partial_{x_i}) = 0$ for all $\delta \in \DerS$, or, equivalently,
\begin{equation}
  \label{eq:flatconnect} \nabla_\delta(\delta') = \sum_i (\delta p_i) \partial_{x_i}
\end{equation}
for $\delta, \delta' \in \DerS$ with $\delta' = \sum p_i \partial_{x_i}$.
Alternatively, this can be characterized by
\begin{align}
  \nabla_\delta(\delta')(\alpha) = \delta(\delta'(\alpha)) \label{eq:defnabla}
\end{align}
for all $\alpha \in V^*$.
Observe that for ~$\nabla$ flat and~$\delta,\delta'$  homogeneous,~\eqref{eq:flatconnect} 
implies that the derivation $\nabla_\delta(\delta')$ is again homogeneous with  
polynomial degree
\begin{equation}
  \label{eq:pdegnabla}\pdeg\big(\nabla_\delta(\delta')\big) = \pdeg(\delta) + \pdeg(\delta') - 1.
\end{equation}

In the sequel, we largely follow the construction of flat systems of invariants as given in~\cite[Sec.~6]{KMS2015} in order to lift the constructions in~\cite{yoshinaga:multicoxeter} to the well-generated case.

\bigskip

As before, we assume in this section that~$W$ is an irreducible well-generated unitary reflection group.
Let $\InvSpecial_1,\ldots,\InvSpecial_\ell$ be the special homogeneous fundamental invariants in $\BBC[\bbx]$ with $\bbx = (x_1,\ldots,x_\ell)$, as given in~\cite[Thm.~6.1]{KMS2015}.
Recall that $\deg\big(\InvSpecial_i\big) = d_i = n_i(V) + 1$ and $\BBC[\InvSpecial_1,\ldots,\InvSpecial_\ell] \cong S^W$.

Consider indeterminates $\bbt = (t_1,\ldots,t_\ell)$ together with the map $t_i \mapsto \InvSpecial_i$ giving an isomorphism 
\[
R \eqdef \BBC[\bbt] \cong \BBC[\InvSpecial_1,\ldots,\InvSpecial_\ell].
\]
Set moreover $\BBC[\bbt'] \eqdef \BBC[t_1,\ldots,t_{\ell-1}]$, its subring generated by $\bbt' = (t_1,\ldots,t_{\ell-1})$.
In order to keep track of the information about the degrees of~$\InvSpecial_1,\ldots,\InvSpecial_\ell$, 
following~\cite[Sec.~6]{KMS2015}, 
we define \defn{weights} of the variables~$t_i$ by
\[
  \weight(t_i) \eqdef \deg(\InvSpecial_i) / h = d_i / h = ( n_i(V) + 1 ) / h.
\]
As usual, set
\begin{equation*}
  \Jtx \eqdef
    \colvecx (t_1,\ldots,t_\ell) = 
    \begin{bmatrix}
      \partial t_1 / \partial x_1 & \cdots & \partial t_\ell / \partial x_1 \\
      \vdots & \ddots & \vdots \\
      \partial t_1 / \partial x_\ell & \cdots & \partial t_\ell / \partial x_\ell
    \end{bmatrix}  \in \BBC[\bbx]^{\ell\times\ell}
\end{equation*}
with inverse matrix $\Jxt \eqdef \Jtx^{-1} = \rowvect^\tr (x_1,\ldots,x_\ell)$.
It is well-known that $\det \Jtx \doteq \prod_{H \in \CA} \alpha_H^{e_H-1}$, see~\cite[Thm.~6.42]{orlikterao:arrangements}.

\medskip

The \defn{primitive vector field}
\[
D \eqdef \partial_{t_\ell} \in \DerR
\]
is given by 
\begin{equation*}
D = \det \Jxt
          \begin{vmatrix}
            \frac{\partial t_1}{\partial x_1} & \cdots & \frac{\partial t_{\ell-1}}{\partial x_1} & \frac{\partial}{\partial x_1} \\
            \vdots & \ddots & \vdots & \vdots \\
            \frac{\partial t_1}{\partial x_\ell} & \cdots & \frac{\partial t_{\ell-1}}{\partial x_\ell} & \frac{\partial}{\partial x_\ell} \\
          \end{vmatrix},
\label{eq:defD}
\end{equation*}
implying in particular that~$D$ is homogeneous with
\begin{equation}
  \label{eq:pdegD}\pdeg(D) = - n_\ell(V) = -(h-1)
\end{equation}
when considered inside $\sum F \partial_{x_i}$ for the fraction field~$F$ of~$S$.
We have seen in \Cref{rem:well-generated} that $h = d_\ell > d_{\ell-1}$.
The primitive vector field~$D$ is thus, up to a non-zero complex constant, independent of the given choice of fundamental invariants.

\bigskip

Consider $X \eqdef V \big/ W = \spec(\BBC[\bbt])$ and let $\hh(\bbt)$ be the \defn{discriminant} of~$W$ given by 
\[
  \hh\big(\InvSpecial_1(\bbx),\ldots,\InvSpecial_\ell(\bbx)\big) = \prod_{H \in \CA} \alpha_H^{e_H}
\]
with vanishing locus $\VanLoc \eqdef \{ \overline p \in X \mid \hh(\overline p) = 0 \}$, \cf~\cite[Def.~6.44]{orlikterao:arrangements}.
Let~$\DerR$ be the $R$-module of logarithmic vector fields, and let
\[
  \DerR(-\log\hh) \eqdef \big\{ \theta \in \DerR \mid \theta\hh \in R\hh \big\}
\]
be the module of logarithmic vector fields along~$\VanLoc$.
We have an $R$-isomorphism between such logarithmic vector fields and $W$-invariant $S$-derivations,
\begin{equation}
  \DerR(-\log\hh) \cong \DerSW, \label{eq:DerRDerSW}
\end{equation}
and $\DerR(-\log\hh)$ is a free $R$-module, \cf~\cite[Cor.~6.58]{orlikterao:arrangements}.

\medskip

Bessis showed in~\cite[Thm.~2.4]{Bes2015} that there exists a \defn{system of flat homogeneous derivations} $\{ \xi_1,\ldots,\xi_\ell \}$ of $\DerR(-\log\hh)$.
This means, its Saito matrix
\[
  M_\xi \eqdef M(\xi_\ell, \ldots, \xi_1) =
  \begin{bmatrix}
    \xi_\ell(t_1) & \cdots & \xi_\ell(t_\ell) \\
    \vdots & \ddots & \vdots \\
    \xi_1(t_1) & \cdots & \xi_1(t_\ell)
  \end{bmatrix}
\]
decomposes as
\begin{align}
  \label{eq:MV}M_\xi = t_\ell \one_\ell + \Mprime(\bbt')
\end{align}
with $\Mprime(\bbt') \in \BBC[\bbt']^{\ell\times\ell}$.
As before, we have $\rowvecXi^\tr = M_\xi \rowvect^\tr$.
Moreover, we obtain that $\hh(\bbt)$ is a monic polynomial in~$t_\ell$ with coefficients in $\BBC[\bbt']$, \ie,
\[
  \hh(\bbt) = t_\ell^\ell + a_{\ell-1}(\bbt')t_\ell^{\ell-1} + \ldots + a_{1}(\bbt')t_\ell + a_{0}(\bbt').
\]
As observed in~\cite[Lem.~3.12]{KMS2015}, such a system of flat homogeneous derivations is unique.
Following~\cite[Eqs.~(52), (53)]{KMS2015}, where this flat system is denoted by $(V_\ell,\ldots,V_1)$, we have
\[
  \weight\big( \xi_{\ell+1-j}(t_i) \big) = 1 - \weight(t_j) + \weight(t_i)  
\]
and
\[
  \xi_1 = \sum \weight(t_i) t_i \partial_{t_i} \in \DerR
\]
is the \defn{Euler vector field} mapped to the (scaled) \defn{Euler derivation}
\begin{equation}
E \eqdef \tfrac{1}{h} \sum x_i \partial_{x_i} \in \DerSW
  \label{eq:Euler}
\end{equation}
under the isomorphism in~\eqref{eq:DerRDerSW}.
As described in~\cite[Lem.~3.9]{KMS2015}, one decomposes
\begin{equation}
  \label{eq:Btilde} M_\xi = \sum \weight(t_i) t_i \tilde B^{(i)}
\end{equation}
and defines the weighted homogeneous $(\ell \times \ell)$-matrix $C(\bbt)$ such that
\begin{equation}
  \tilde B^{(i)} = \partial C \big/ \partial t_i \quad \text{and} \quad \xi_1C = M_\xi.
  \label{eq:Cmat}
\end{equation}
In this case,~\cite[Thm.~6.1]{KMS2015} yields that $t_i = C_{\ell,i}$ and thus,~$\bbt$  is a \defn{flat coordinate system} on~$X$ associated to the \emph{Okubo type differential equation}
\begin{equation}
  d\Yprime = -M_\xi^{-1}\ dC\ \Binf\ \Yprime,
  \label{eq:differentialequation}
\end{equation}
where~$B_\infty$ is the diagonal matrix
\[
  \Binf \eqdef \diag\big( \weight(t_i) - (h+1) / h \big) = \diag\big( ( d_i - h - 1 ) / h \big),
\]
and
\begin{equation*}
  \Yprime \eqdef -\Binf^{-1} \rowvecXi^\tr (x_1,\ldots,x_\ell)
          = -\Binf^{-1} M_\xi \Jxt.
\end{equation*}
Define a connection~$\nabla$ on~$\DerR$ by
\begin{equation*}
  \nabla \colvect = -M_\xi^{-1}(\bbt) dC(\bbt)(B_\infty + \one_\ell) \colvect,
\end{equation*}
where $dC = \sum \tilde B^{(i)} dt_i$ is the differential of the matrix~$C(\bbt)$ as given in~\eqref{eq:Cmat}.

\begin{proposition}
\label{prop:flatextension}
  The connection~$\nabla$ extends to a connection on~$\DerS$ which is \emph{flat}, \ie,
  \[
    \nabla \colvecx = 0.
  \]
\end{proposition}

\begin{proof}
  Using the definition of~$\nabla$ and the Leibniz rule, we obtain
  \begin{equation}
    \begin{aligned}
      \nabla\colvect
        &= d\Jxt\ \colvecx + \Jxt\ \nabla\colvecx \\
        &= -M_\xi^{-1}\ dC\ (\Binf + \one_\ell )\ \Jxt\ \colvecx.
    \end{aligned}
    \label{eq:flatproof1}
  \end{equation}
  By~\eqref{eq:differentialequation}, we have
  \begin{equation}
    \begin{aligned}
      d\Yprime &= - \Binf^{-1}\ \big( dM_\xi\ \Jxt + M_\xi\ d\Jxt \big)\\
               &= -M_\xi^{-1}\ dC\ \Binf\ \Yprime \\
               &= dC\ \Jxt,
    \end{aligned}
    \label{eq:flatproof2}
  \end{equation}
  where all $\tilde B^{(i)}$ and $M_\xi$ mutually commute, according to~\cite[Eq.~(13)]{KMS2015}.
  Thanks to~\cite[Eq.~(28)]{KMS2015}, we have
  \[
    dM_\xi = dC + [dC, \Binf].
  \]
  The identity \eqref{eq:flatproof2} then implies
  \begin{equation}
    \begin{aligned}
      -M_\xi\ d\Jxt &= \Binf\ dC\ \Jxt + \big(dC + [dC, \Binf]\big) \Jxt \\
                    &= dC\ (\Binf + \one_\ell) \Jxt.
    \end{aligned}
    \label{eq:flatproof3}
  \end{equation}
  We finally deduce from~\eqref{eq:flatproof1} and~\eqref{eq:flatproof3} that
  \[
     \Jxt\ \nabla\colvecx = -\Big(M_\xi^{-1}\ dC\ (\Binf + \one_\ell) \Jxt + d\Jxt\Big) \colvecx = 0.
  \]
  Since~$\Jxt$ is invertible, the result follows.
\end{proof}

One further main ingredient in the proof of \Cref{thm:main1} is the following proposition, where we recall $\Mprime(\bbt') = M_\xi - t_\ell \one_\ell$ from~\eqref{eq:MV}.

\begin{proposition}
\label{prop:mainingredient}
  We have $\BBC[\bbt']$-isomorphisms
  \begin{equation*}
    \begin{aligned}
    \prim &: \DerR(-\log\hh) \longrightarrow \DerR \\
    \priminv &: \DerR \longrightarrow \DerR(-\log\hh)
    \end{aligned}\label{eq:nablainv}
  \end{equation*}
  given by
  \begin{equation}
    \begin{aligned}
      \prim\colvect    &= -M_\xi^{-1}(B_\infty + \one_\ell) \colvect \\
      \priminv\colvect &= -B^{-1}_\infty \colvecXi = -B^{-1}_\infty M_\xi \colvect.
    \end{aligned}
    \label{eq:primiso}
  \end{equation}
  Moreover, for $k > 0$ we get
  \begin{equation}
  \label{eq:computenablaDinv}
    \priminv(t_\ell^k \partial_{t_{i}}) = \tfrac{h}{(k+1)h+1-d_i}\Big( t_\ell^k \xi_{\ell+1-i} - k \sum_{j=1}^\ell \big[\Mprime(\bbt')\big]_{ij} \priminv(t_\ell^{k-1}\partial_{t_j})\Big).
  \end{equation}
\end{proposition}

\begin{proof}
  One first directly calculates that~$\nabla_D$ is $\BBC[\bbt']$-linear.
  The first equation in~\eqref{eq:primiso} is a direct consequence of the fact that~$\tilde B^{(\ell)} = \one_\ell$ which follows
  from ~\eqref{eq:Btilde} in light of~\eqref{eq:MV}.
  On the other hand, we have
  \begin{equation}
  \label{eq:nablaDdescription}
    \begin{aligned}
    \nabla_D \colvecXi
      &= \nabla_D\left( M_\xi \colvect\right) \\
      &= \frac{\partial M_\xi}{\partial t_\ell} \colvect + M_\xi \nabla_D \colvect \\
      &= \colvect - (\Binf + \one_\ell) \colvect \\
      &= -\Binf \colvect,
    \end{aligned}
  \end{equation}
  where we used that $\frac{\partial M_\xi}{\partial t_\ell} = \one_\ell$, see again~\eqref{eq:MV}.
  Moreover,
  \begin{equation}
  \label{eq:nableDdescription2}
    \begin{aligned}
    \prim(t_\ell^k \xi_{\ell+1-i}) &= k t_\ell^{k-1} \xi_{\ell+1-i} +t_\ell^k \prim(\xi_{\ell+1-i}) \\
                     &= k t_\ell^{k-1}\Big( t_\ell \partial_{t_i} + \sum_{j=1}^\ell \big[\Mprime(\bbt')\big]_{ij} \partial_{t_j} \Big) + \tfrac{h+1-d_i}{h} t_\ell^k \partial_{t_i} \\
                     &= \tfrac{(k+1)h+1-d_i}{h} t_\ell^k \partial_{t_i} + k \sum_{j=1}^\ell \big[\Mprime(\bbt')\big]_{ij} t_\ell^{k-1}\partial_{t_j},
    \end{aligned}
  \end{equation}
  implying that $\nabla_D$ is indeed bijective, where we used that $(k+1)h+1-d_i > 0$, and thus a $\BBC[\bbt']$-isomorphism.
  Applying now $\nabla_D^{-1}$ to~\eqref{eq:nablaDdescription} yields the second equation in~\eqref{eq:primiso}, and rewriting~\eqref{eq:nableDdescription2} as
  \begin{equation*}
    t_\ell^k \partial_{t_i} = \tfrac{h}{(k+1)h+1-d_i} \big( \prim(t_\ell^k \xi_{\ell+1-i}) - k \sum_{j=1}^\ell \big[\Mprime(\bbt')\big]_{ij} t_\ell^{k-1}\partial_{t_j} \big),
  \end{equation*}
  and applying $\nabla_D^{-1}$ yields~\eqref{eq:computenablaDinv}.
\end{proof}

Using \Cref{prop:mainingredient}, we obtain the analogue of the \emph{Hodge filtration} for reflection arrangements of well-generated unitary reflection groups introduced by Saito for Coxeter arrangements in~\cite{Sai1993}, compare also~\cite{terao:hodge}.
Let $\Hodgedec_0$ be the $\BBC[\bbt']$-submodule of $\Der_R$ generated by $\partial_{t_1},\ldots,\partial_{t_\ell}$ and let $\Hodgedec_k \eqdef \nabla_D^{-k}(\Hodgedec_0)$ for $k \in \BBN$.
In particular, we see by \Cref{prop:mainingredient} that $\Hodgedec_1$ coincides with the $\BBC[\bbt']$-submodule generated by $\xi_1,\ldots,\xi_\ell$.
Then, using~\eqref{eq:computenablaDinv}, we confirm that
\begin{equation*}
  \Der_R = \bigoplus_{k=0}^\infty \Hodgedec_k \quad\text{and}\quad \Hodgedec_0 \oplus \cdots \oplus \Hodgedec_k = \Hodgedec_0 \oplus t_\ell\Hodgedec_0 \oplus \cdots \oplus t_\ell^k\Hodgedec_0.
\end{equation*}
Note that $\Hodgedec_k \neq t_\ell^k\Hodgedec_0$  in general. However, we obtain the (increasing) \defn{Hodge filtration} of $\Der_R$ defined by
\begin{equation*}
  \Hodgefil_k \eqdef \bigoplus_{i=0}^k \Hodgedec_i = \bigoplus_{i=0}^k t_\ell^i\Hodgedec_0.
\end{equation*}
From this filtration, we can now derive the \emph{universality} of $\priminvm(E)$ for the Euler derivation~$E$ defined in~\eqref{eq:Euler}.

\begin{proposition}
\label{prop:univerality}
  Let $\theta_1,\ldots,\theta_\ell \in \Der_S$ be linearly independent over~$S$.
  Then
  \[
    \nabla_{\theta_1}\prim^{-m}(E),\ldots,\nabla_{\theta_\ell}\prim^{-m}(E)
  \]
  are linearly independent over~$S$.
\end{proposition}

\begin{proof}
  It is well-known that $\det \Jtx \doteq \prod_{H \in \CA} \alpha_H^{e_H-1} \neq 0$ implies that it is sufficient to prove linear independence of $\nabla_{\partial_{t_1}}\prim^{-m}(E),\ldots,\nabla_{\partial_{t_\ell}}\prim^{-m}(E)$ over~$R$.
  For the sake of contradiction, assume that
  \begin{equation}
    \sum_{i=1}^\ell a_i \nabla_{\partial_{t_i}}\prim^{-m}(E) = 0
    \label{eq:lineardep}
  \end{equation}
  with $a_i \in R = \BBC[\bbt]$.
  Write $a_i  = c_{i0} + c_{i1}t_\ell + c_{i2}t_\ell^2 + \cdots + c_{id}t_\ell^d$ for $c_{ij} \in \BBC[\bbt']$ where~$d$ is the maximal degree of $t_\ell$ among all $a_i$'s.
  In particular, $c_{id} \neq 0$ for some~$i$.
  We easily find that $\nabla_{\partial_{t_i}}(E) = \tfrac{1}{h}\partial_{t_i} \in \Hodgedec_0$ and thus $\nabla_{\partial_{t_i}}\priminvm(E) = \priminvm\nabla_{\partial_{t_i}}(E) \in \Hodgedec_m$.
  Hence, we also have $\sum_{i=1}^\ell c_{id} \nabla_{\partial_{t_i}}\priminvm(E) \in \Hodgedec_m$.
  On the other hand, we obtain from~\eqref{eq:lineardep} that the coefficient of the degree~$d+m$ term with respect to $t_\ell$ vanishes, and therefore, $\sum_{i=1}^\ell c_{id} \nabla_{\partial_{t_i}}\priminvm(E) \in \Hodgefil_{m-1}$, implying that this sum also vanishes.
  Applying $\nabla_D^{m}$, we finally have
  \[
    0 = \nabla_D^{m}\Big(\sum_{i=1}^\ell c_{id} \nabla_{\partial_{t_i}}\priminvm(E)\Big) = \sum_{i=1}^\ell c_{id} \nabla_{\partial_{t_i}}(E) = \frac{1}{h} \sum_{i=1}^\ell c_{id} \partial_{t_i}.
  \]
  Therefore, $c_{1d} = \cdots = c_{\ell d} = 0$, contradicting the fact that $c_{id} \neq 0$ for some~$i$.
\end{proof}
  
After having established the universality of $\priminvm(E)$, the following is our generalization of \cite[Thm.~7]{yoshinaga:multicoxeter} to the well-generated setting.

\begin{theorem}
\label{thm:main1strong}
  Let~$W$ be an irreducible, well-generated unitary reflection group with reflection arrangement $\CA$.
  Let $\omega : \CA \to \BBN$ given by $\omega(H) = e_H$, and let $m \in \BBN$.
  Suppose that $\mu : \CA \rightarrow \{0,1\}$ such that $\FD(\CA,\mu)$ is free with homogeneous basis $\theta_1,\ldots,\theta_\ell$.
  Then $\FD(\CA,m\omega+\mu)$ is free with basis
  \[
    \nabla_{\theta_1}\prim^{-m}(E),\ldots,\nabla_{\theta_\ell}\prim^{-m}(E).
  \]
  Moreover,
  \[
    \exp\big(\CA,m\omega+\mu\big) = \big\{ mh + \pdeg(\theta_1), \ldots, mh + \pdeg(\theta_\ell) \big\}.
  \]
\end{theorem}

Armed with \Cref{thm:main1strong}, we derive our first main theorem.

\begin{proof}[Proof of \Cref{thm:main1}]
  One obtains the two statements in the theorem from the special cases in \Cref{thm:main1strong} with
    $\mu \equiv 0$ and $\mu \equiv 1$.
    Freeness in the first case is trivial, and is due to Terao~\cite{terao:freereflections} in the second.
\end{proof}

\begin{proof}[Proof of \Cref{thm:main1strong}]
  Let $\delta \in \DerS$ and $\alpha = \alpha_H$ with $H \in \CA$.
  We first show that, for any~$m \in \BBN$,
  \begin{equation}
    \prim(\delta) \alpha \in \alpha^m S \quad \Longleftrightarrow \quad \delta\alpha \in \alpha^{m+e_H} S. \label{eq:lifting}
  \end{equation}
  For the reverse implication, suppose that $\delta\alpha = \alpha_H^{k+e_H} f$ for some~$f \in S$ and $k \in \BBN$.
  We then obtain from~\eqref{eq:defnabla} that
  \begin{equation*}
    \prim(\delta)\alpha = D(\delta\alpha) 
                        \doteq \det \Jxt
                              \begin{vmatrix}
                                \frac{\partial t_1}{\partial x_1} & \cdots & \frac{\partial t_{\ell-1}}{\partial x_1} & \frac{\partial \alpha^{k+e_H}f}{\partial x_1} \\
                                \vdots & \ddots & \vdots & \vdots \\
                                \frac{\partial t_1}{\partial x_\ell} & \cdots & \frac{\partial t_{\ell-1}}{\partial x_\ell} & \frac{\partial \alpha^{k+e_H}f}{\partial x_\ell}
                              \end{vmatrix}.
  \end{equation*}
  It now follows from the product rule for derivations that $\prim(\delta)\alpha$ is divisible by $\alpha^{k}$.

  For the forward implication, assume that~$k$ is maximal such that $\delta\alpha = \alpha^{k+e_H} f$.
  We show that in this case, $\prim(\delta) \alpha \notin \alpha^{k+1} S$.
  We may assume, after a possible change of basis, that $\alpha = x_\ell$.
  Since $\det \Jxt = \det \Jtx^{-1} = \big(\prod_{H \in \CA} \alpha_H^{e_H-1} \big)^{-1}$, 
  we have to show that the maximal minor
  \begin{equation*}
    \begin{vmatrix}
      \frac{\partial t_1}{\partial x_1} & \cdots & \frac{\partial t_{\ell-1}}{\partial x_1} \\
      \vdots & \ddots & \vdots \\
      \frac{\partial t_1}{\partial x_{\ell-1}} & \cdots & \frac{\partial t_{\ell-1}}{\partial x_{\ell-1}}
    \end{vmatrix}
    \label{eq:partialmat}
  \end{equation*}
  is not divisible by~$\alpha$.
  This follows from a variant of the argument in the proof of~\cite[Lem.~6.41]{orlikterao:arrangements}.
  Arguing as in \loccit, the sequence $(h_1,\ldots,h_\ell) = (t_1(\bbx), \ldots, t_{\ell-1}(\bbx), x_\ell)$ is regular.
  Because the considered determinant equals
  \begin{equation*}
    \begin{vmatrix}
      \frac{\partial t_1}{\partial x_1} & \cdots & \frac{\partial t_{\ell-1}}{\partial x_1} & \frac{\partial x_\ell}{\partial x_1} \\
      \vdots & \ddots & \vdots & \vdots \\
      \frac{\partial t_1}{\partial x_{\ell-1}} & \cdots & \frac{\partial t_{\ell-1}}{\partial x_{\ell-1}} & \frac{\partial x_\ell}{\partial x_{\ell-1}}\\[10pt]
      \frac{\partial t_1}{\partial x_{\ell}} & \cdots & \frac{\partial t_{\ell-1}}{\partial x_{\ell}} & \frac{\partial x_\ell}{\partial x_\ell}\\
    \end{vmatrix},
  \end{equation*}
  applying \loccit\ directly shows that this determinant does not belong to the ideal generated by $(t_1(\bbx),\ldots,t_{\ell-1}(\bbx),x_\ell)$.
  In particular, the determinant is not divisible by $x_\ell = \alpha$, as desired.

  \medskip

  Next, observe that~\eqref{eq:lifting} and \Cref{prop:mainingredient} immediately imply
  \[
    \delta \alpha \in \alpha^k S \quad \Longleftrightarrow \quad \priminv(\delta)\alpha \in \alpha^{k+e_H} S
  \]
  for $\delta \in \DerR$, forcing $\priminvm(E)\alpha = \alpha^{me_H+1}f$ for some $f \in S$.
  Thus, applying $\nabla_\theta$ for $\theta \in \FD(\CA,\mu)$ to both sides and using~\eqref{eq:defnabla} entails
  \[
    \nabla_\theta\priminvm(E)\alpha = \alpha^{me_H}\big( (me_H + 1)(\theta\alpha)f + \alpha(\theta(f) \big).
  \]
  As $\theta\alpha$ is divisible by $\alpha^{\mu(H)}$ and $0 \leq \mu(H) \leq 1$, we obtain that $\nabla_\theta\priminvm(E)\alpha$ is divisible by $\alpha^{me_H+\mu(H)}$, implying that 
  \[
    \nabla_\theta\priminvm(E) \in \FD(\CA,m\omega+\mu).
  \]
  For ~$\theta \in \FD(\CA,\mu)$ homogeneous, we obtain from~\eqref{eq:pdegnabla} and~\eqref{eq:pdegD} that $\nabla_{\theta}\priminvm(E)$ is homogeneous as well with 
  \[
    \pdeg\big(\nabla_{\theta}\priminvm(E)\big) = mh + \pdeg\theta.
  \]
  Let now $\theta_1,\ldots,\theta_\ell$ be the given homogeneous basis of $\FD(\CA,\mu)$.
  Then, since $\sum \pdeg(\theta_i) = |\mu|$, we immediately get
  \[
    \sum_{i=1}^\ell \pdeg\big(\nabla_{\theta_i}\priminvm(E)\big) = mh\ell + |\mu| = | m\omega + \mu|.
  \]
  The statement then follows from the universality of $\priminvm(E)$ in \Cref{prop:univerality} using \Cref{thm:zieglersaito}\eqref{eq:zieglersaito3}.
\end{proof}

\subsection{An example}

We finish this section with a detailed example of the computation of the basis for $\FD(\CA(W),\omega)$ with $W = G(3,1,2)$.
In this cases, the degrees are
\[
d_1 = 3, \quad d_2 = h = 6.
\]

We refer to~\cite[Rem.~6.2]{KMS2015} for a general strategy how to compute a flat system of invariants from the \emph{potential vector field} corresponding to the Okubo type differential equation~\eqref{eq:differentialequation} as defined in~\cite[Def~4.2]{KMS2015}.
Such have been computed in many types in~\cite{AL2016}, see also~\cite{KMS20151}.

\medskip

Given such a potential vector field $\vec{g} = (g_1(\bbt),\ldots,g_\ell(\bbt))$ and a flat system of invariants $\InvSpecial_1(\bbx),\ldots,\InvSpecial_\ell(\bbx)$, the general strategy is as follows:
\begin{enumerate}[(1)]
\setlength\itemsep{10pt}
  \item Compute $\tilde B^{(i)}$ using $[\tilde B^{(i)}]_{jk} = \frac{\partial^2g_k}{\partial t_i \partial t_j}$, as given in the proof of~\cite[Prop.~4.4]{KMS2015}.
  \item Compute $M_\xi = \sum \weight(t_i) t_i \tilde B^{(i)}$, as given in~\eqref{eq:Btilde}.
  \item Compute $\prim^{-m}(E) \in \DerR(-\log\hh)$, using \Cref{prop:mainingredient}.
  \item Transfer $\prim^{-m}(E) \in \DerR(-\log\hh)$ into $\prim^{-m}(E) \in \DerSW$ by specializing $t_i \mapsto \InvSpecial_i(\bbx)$ and using
  \[
    \rowvect^\tr = \Jxt\ \rowvecx^\tr.
  \]
  \item Given a homogeneous basis $\theta_1,\ldots,\theta_\ell$ of  $\FD(\CA,\mu)$ for some $\mu : \CA \rightarrow \{0,1\}$, one finally uses \Cref{prop:flatextension} to compute the homogeneous basis of $\FD(\CA,m\omega+\mu)$.
\end{enumerate}

\medskip

Following~\cite[Sec.~5.17]{AL2016} for $G(3,1,2)$, the potential vector field $\vec{g} = \big(g_1(\bbt),g_2(\bbt) \big)$ is given by
\begin{equation*}
  g_1(\bbt) = \tfrac{1}{18} t_1^3 + t_1 t_2, \qquad
  g_2(\bbt) = \tfrac{1}{54} t_1^4 + \tfrac{1}{2} t_2^2
\end{equation*}
and a flat system of fundamental invariants is given by
\begin{equation*}
  \InvSpecial_1(\bbx) = x_1^3 + x_2^3, \qquad
  \InvSpecial_2(\bbx) = \tfrac{1}{6} x_1^6 - \tfrac{5}{3} x_1^3 x_2^3 + \tfrac{1}{6}x_2^6.
\end{equation*}

\medskip

First, we obtain from the degrees that
\[
  -\Binf = \begin{bmatrix}
            \frac{2}{3} & 0 \\
            0 & \frac{1}{6}
          \end{bmatrix}.
\]
From the potential vector field, we compute
\begin{equation*}
  \tilde B^{(1)} = \begin{bmatrix}
                     \frac{1}{3} t_1 & \frac{2}{9} t_1^2 \\
                     1 & 0
                   \end{bmatrix},
  \quad
  \tilde B^{(2)} = \begin{bmatrix}
                     1 & 0 \\
                     0 & 1
                   \end{bmatrix}
\end{equation*}
implying
\[
  M_\xi = \tfrac{1}{2} t_1 \tilde B^{(1)} + t_2 \tilde B^{(2)} = \begin{bmatrix}
                                                                   \frac{1}{6} t_{1}^{2} + t_{2} & \frac{1}{9} t_{1}^{3} \\
                                                                   \frac{1}{2} t_{1} & t_{2}.
                                                                 \end{bmatrix}
\]
Next, we compute
\begin{equation*}
  \priminv \pmat{\partial_{t_1} \\ \partial_{t_2}} = -B^{-1}_\infty M_\xi \pmat{\partial_{t_1} \\ \partial_{t_2}}=
        \begin{bmatrix}
          \frac{1}{4} t_{1}^{2} + \frac{3}{2} t_{2} & \frac{1}{6} t_{1}^{3} \\
          3 t_{1} & 6 t_{2}
        \end{bmatrix}\pmat{\partial_{t_1} \\ \partial_{t_2}}
\end{equation*}
and
\begin{align*}
  \priminv(t_2 \partial_{t_2})
      &= \tfrac{6}{7}\Big( t_2 \xi_1 - \big([\Mprime]_{21} \priminv(\partial_{t_1}) + [\Mprime]_{22} \priminv(\partial_{t_2}) \big) \Big) \\
      &= \tfrac{6}{7}\big( t_2 ( \tfrac{1}{2} t_1\partial_{t_1} + t_2 \partial_{t_2} ) -  \tfrac{1}{2}t_1 \priminv(\partial_{t_1}) \big)\\
      &= \tfrac{3}{7}t_1t_2\partial_{t_1} + \tfrac{6}{7}t_2^2\partial_{t_2} - \tfrac{3}{7}t_1 \big((\tfrac{1}{4} t_{1}^{2} + \tfrac{3}{2} t_{2})\partial_{t_1} + \tfrac{1}{6} t_{1}^{3} \partial_{t_2} \big)\\
      &= -\big(\tfrac{3}{28} t_{1}^{3} + \tfrac{3}{14} t_{1} t_{2} \big) \partial_{t_1} +  \big(\tfrac{6}{7} t_{2}^{2}-\tfrac{1}{14} t_{1}^{4} \big)\partial_{t_2}.
\end{align*}
Therefore, we have
\begin{align*}
  \priminv(E) &= \priminv(\tfrac{1}{2}t_1\partial_{t_1} + t_2\partial_{t_2} ) \\
              &= \tfrac{1}{2}t_1 \priminv(\partial_{t_1}) + \priminv(t_2\partial_{t_2}) \\
              &= \tfrac{1}{2}t_1 \big((\tfrac{1}{4} t_{1}^{2} + \tfrac{3}{2} t_{2})\partial_{t_1} + \tfrac{1}{6} t_{1}^{3} \partial_{t_2}\big) -\big(\tfrac{3}{28} t_{1}^{3} + \tfrac{3}{14} t_{1} t_{2} \big) \partial_{t_1} +  \big(\tfrac{6}{7} t_{2}^{2}-\tfrac{1}{14} t_{1}^{4} \big)\partial_{t_2}\\
              &= (\tfrac{1}{56} t_{1}^{3} + \tfrac{15}{28} t_{1} t_{2})\partial_{t_1} + (\tfrac{1}{84} t_{1}^{4} + \tfrac{6}{7} t_{2}^{2})\partial_{t_2} \in \DerR(-\log\hh).
\end{align*}
We next compute
\[
  \Jtx = \begin{bmatrix}
           3 x_{1}^{2} & x_{1}^{5} - 5 x_{1}^{2} x_{2}^{3} \\
           3 x_{2}^{2} & -5 x_{1}^{3} x_{2}^{2} + x_{2}^{5}
         \end{bmatrix},
         \qquad
         \det\Jtx = -18 x_{1}^{5} x_{2}^{2} + 18 x_{1}^{2} x_{2}^{5}
\]
and
\[
  \Jxt = \Jtx^{-1} = \det(\Jtx)^{-1} \begin{bmatrix}
                                      -5 x_{1}^{3} x_{2}^{2} + x_{2}^{5} & 5 x_{1}^{2} x_{2}^{3} - x_{1}^{5} \\
                                      -3 x_{2}^{2} & 3 x_{1}^{2}
                                    \end{bmatrix}
\]
to obtain
\[
  \priminv(E) = (\tfrac{1}{28} x_{1}^{7} - \tfrac{1}{4} x_{1}^{4} x_{2}^{3}) \partial_{x_1} + ( \tfrac{1}{28} x_{2}^{7} - \tfrac{1}{4} x_{1}^{3} x_{2}^{4} )\partial_{x_2} \in \DerSW.
\]
We finally obtain
\begin{align*}
  \Theta_1 \eqdef \nabla_{\partial_{x_1}}\priminv(E)
    &= \frac{\partial (\tfrac{1}{28} x_{1}^{7} - \tfrac{1}{4} x_{1}^{4} x_{2}^{3})}{\partial {x_1}} \partial_{x_1} + \frac{\partial ( \tfrac{1}{28} x_{2}^{7} - \tfrac{1}{4} x_{1}^{3} x_{2}^{4} )}{\partial x_1}\partial_{x_2} \\
    &= (\tfrac{1}{4} x_{1}^{6} -  x_{1}^{3} x_{2}^{3}) \partial_{x_1} -\tfrac{3}{4} x_{1}^{2} x_{2}^{4} \partial_{x_2} \\[15pt]
  \Theta_2 \eqdef \nabla_{\partial_{x_2}}\priminv(E)
    &= \frac{\partial (\tfrac{1}{28} x_{1}^{7} - \tfrac{1}{4} x_{1}^{4} x_{2}^{3})}{\partial {x_2}} \partial_{x_1} + \frac{\partial ( \tfrac{1}{28} x_{2}^{7} - \tfrac{1}{4} x_{1}^{3} x_{2}^{4} )}{\partial x_2}\partial_{x_2} \\
    &= -\tfrac{3}{4} x_{1}^{4} x_{2}^{2} \partial_{x_1} + (\tfrac{1}{4} x_{2}^{6}- x_{1}^{3} x_{2}^{3}) \partial_{x_2}.
\end{align*}
One can easily check that $\{\Theta_1,\Theta_2\}$ is indeed a homogeneous basis of $\FD(\CA,\omega)$.

\section{Proof of \Cref{thm:main2}}
\label{sec:proofmain2}

In this section, we prove in \Cref{thm:main2strong} a strengthened version of \Cref{thm:main2} for the imprimitive groups $G(de,e,\ell)$ with
\[
  r \eqdef de \geq 2 \quad \text{and} \quad \ell \geq 2.
\]
We fix these parameters throughout.
This restriction means we exclude the symmetric groups $G(1,1,\ell)$ and the cyclic groups $G(d,1,1)$ from our subsequent considerations.
The first has been treated in \cite{terao:multi}, the second is trivial.

\medskip

Recall that the simple reflection arrangements in the considered cases are given by
\[
  Q(\CA) = \begin{cases}
             (x_1 \cdots x_\ell ) \prod_{1\le i<j \le \ell} (x_i^r-x_j^r) &\text{ if } d > 1, \\[5pt]
             \phantom{(x_1 \cdots x_\ell )} \prod_{1\le i<j \le \ell} (x_i^r-x_j^r) &\text{ if } d = 1,
           \end{cases}
\]
see~\cite[Sec.~6.4]{orlikterao:arrangements}.
Moreover,
\begin{align*}
  e_H &= d \quad \text{ for } H = \ker(x_i) \text{ with } 1 \leq i \leq \ell, \\
  e_H &= 2 \quad \text{ for } H = \ker(x_i - \zeta x_j) \text{ with } 1 \leq i < j \leq \ell \text{ and } \zeta^r = 1.
\end{align*}

The following theorem is our more general version of \Cref{thm:main2}.
\begin{theorem}
\label{thm:main2strong}
  Let $m,m_1,\ldots,m_\ell, \in \mathbb{N}$ such that $q\eqdef\lfloor (m_i-1)/r \rfloor$ does not depend on~$i$.
  Set $a \eqdef (\ell-1)r, m' \eqdef \sum m_i$, and $c \eqdef ma +qr + 1$.
  \begin{enumerate}[(i)]
    \item\label{eq:main2strong1} The multi-arrangement $(\CA,\mu)$ with defining polynomial
      \begin{equation*}
        Q(\CA,\mu) = x_1^{m_1} \cdots x_\ell^{m_\ell} \prod_{1\le i<j \le \ell} (x_i^r-x_j^r)^{2m}
      \end{equation*}
      is free with exponents
      \begin{equation*}
        \exp(\CA,\mu) = \big\{ ma+m_1,\ldots,ma+m_\ell \big\}. 
      \end{equation*}

    \item\label{eq:main2strong2} The multi-arrangement $(\CA,\mu)$ with defining polynomial
      \begin{equation*}
        Q(\CA,\mu) = x_1^{m_1} \cdots x_\ell^{m_\ell} \prod_{1\le i<j \le \ell} (x_i^r-x_j^r)^{2m+1}
      \end{equation*}
      is free with exponents
      \begin{equation*}
        \begin{aligned}
          \exp(\CA,\mu) &= \Big\{ c + m' - \ell( qr + 1), \quad c + r, c + 2r, \ldots, c + (\ell-1)r \Big\} \\
                        &= \Big\{ (m-q)a +  m' - \ell + 1, \\
              &\hspace*{24.5pt} ma + (q+1)r + 1, ma + (q+2)r + 1, \ldots, ma + (q+\ell-1)r + 1 \Big\}.
        \end{aligned}        
      \end{equation*}
  \end{enumerate}
\end{theorem}
In~\eqref{eq:main2strong2}, we provide two alternative formulas for later reference.
We prove the two parts of this theorem in \Cref{sec:thmmain2strong1,,sec:thmmain2strong2}, respectively.

\medskip

Armed with \Cref{thm:main2strong}, we can deduce our second main result, \Cref{thm:main2}.
We treat the three cases $d=1$, $e=1$, and $d,e \geq 2$ separately, 
and observe that the first two are well-generated while the third is not.

\begin{proof}[Proof of \Cref{thm:main2}~\eqref{eq:main21}]
  For $d=1$, we have Coxeter number $h = (\ell-1)e$.
  Consider the defining polynomial
  \[
    Q(\CA,m\omega) = \prod_{1\le i<j \le \ell} (x_i^{r}-x_j^{r})^{2m}.
  \]
  This is the case $m_1=m_2=\ldots=m_\ell=0$ in \Cref{thm:main2strong}\eqref{eq:main2strong1}.
  Thus,~$(\CA,m\omega)$ is free  with
  \[
    \exp(\CA,m\omega) = \big\{ m(\ell-1)e,\ldots,m(\ell-1)e\big\} = \big\{ mh,\ldots,mh\big\},
  \]
  as claimed.

  \medskip

  For $e=1$, we have Coxeter number $h = \ell d$.
  Consider the defining polynomial
  \[
    Q(\CA,m\omega) = x_1^{rm} \cdots x_\ell^{rm} \prod_{1\le i<j \le \ell} (x_i^{r}-x_j^{r})^{2m}.
  \]
  This is the case $m_1=m_2=\ldots=m_\ell = rm$ in \Cref{thm:main2strong}\eqref{eq:main2strong1}.
  Thus,~$(\CA,m\omega)$ is free  with
  \[
    \exp(\CA,m\omega) = \big\{ m\ell r,\ldots,m\ell r \big\} = \big\{ mh,\ldots,mh\big\},
  \]
  as claimed.

  \medskip

  For $d,e \geq 2$, we have Coxeter number $h=(\ell-1)r+d$.
  Consider the defining polynomial
  \begin{equation*}
    Q(\CA,m\omega) = x_1^{dm} \cdots x_\ell^{dm} \prod_{1\le i<j \le \ell} (x_i^{r}-x_j^{r})^{2m}. \label{eq:Grpleven}
  \end{equation*}
  This is the case $m_1=m_2=\ldots=m_\ell=dm$ in \Cref{thm:main2strong}\eqref{eq:main2strong1}.
  Thus,~$(\CA,m\omega)$ is free  with
  \[
    \exp(\CA,m\omega) = \big\{ m((\ell-1)r+d),\ldots,m((\ell-1)r+d) \big\} = \big\{ mh,\ldots,mh\big\},
  \]
  as claimed.
\end{proof}

\begin{proof}[Proof of \Cref{thm:main2}~\eqref{eq:main22}]
  For $d=1$, we have Coxeter number $h = (\ell-1)r$, and
  \[
     \big\{n_1(V^*),\ldots,n_\ell(V^*)\big\} = \big\{ (\ell-1)r-\ell+1,\quad 1,r+1,2r+1,\ldots,(\ell-2)r+1\big\}.
  \]
  Consider the defining polynomial
  \[
    Q(\CA,m\omega+1) = \prod_{1\le i<j \le \ell} (x_i^{r}-x_j^{r})^{2m+1}.
  \]
  This is the case $m_1=m_2=\ldots=m_\ell=0$ in \Cref{thm:main2strong}\eqref{eq:main2strong2}.
  We have $a=(\ell-1)r, q=-1$, and $m'=0$, and $(\CA,m\omega+1)$ is free with
  \begin{align*}
    \exp(\CA,&m\omega +1) \\
         &=
                 \Big\{ (m+1)(\ell-1)r-\ell+1, \\
                &\quad\quad m(\ell-1)r + 1,m(\ell-1)r + r + 1, \ldots, m(\ell-1)r + (\ell-2)r + 1 \Big\} \\
                &= \Big\{ mh+(\ell-1)r-\ell+1,\quad mh+1,mh+r+1,\ldots,mh+(\ell-2)r+1 \Big\}\\
                &= \big\{ mh+n_1(V^*),\ldots,mh+n_\ell(V^*) \big\},
  \end{align*}
  as claimed.

  \medskip

  For $e=1$, we have Coxeter number $h = \ell d$ and
  \[
    \big\{ n_1(V^*),\ldots,n_\ell(V^*)\big\} = \big\{ 1,d+1,2d+1,\ldots,(\ell-1)d+1\big\}.
  \]
  Consider the defining polynomial
  \[
    Q(\CA,m\omega+1) = x_1^{dm+1} \cdots x_\ell^{dm+1} \prod_{1\le i<j \le \ell} (x_i^{d}-x_j^{d})^{2m+1}.
  \]
  This is the case $r=d, m_1=m_2=\ldots=m_\ell = dm+1$ in \Cref{thm:main2strong}\eqref{eq:main2strong2}.
  We have $a=(\ell-1)d,q=m$, and $m'=\ell dm+\ell$, and $(\CA,m\omega+1)$ is free with 
  \begin{align*}
    \exp(\CA,m\omega+1) &= 
      \Big\{ \ell dm+\ell-\ell+1, \\
                &\hspace*{40pt} m(\ell-1)d + (m+1)d + 1, \ldots, m(\ell-1)d + (m+\ell-1)d + 1 \Big\} \\
                &= \Big\{ mh+1,mh+d+1,\ldots,mh+(\ell-1)d+1 \Big\}\\
                &= \big\{ mh+n_1(V^*),\ldots,mh+n_\ell(V^*) \big\},
  \end{align*}
  as claimed.

  \medskip

  For $d,e \geq 2$, we have Coxeter number $h=(\ell-1)de+d$, and, as computed in~\cite[Sec.~3]{gordongriffeth}, we have $\ord(\Psi) = e$ with
  \begin{multline}
    \big\{ n_1(\Psi^{s}(V^*)^*), \ldots , n_\ell(\Psi^{s}(V^*)^*)\big\} = \\
    \Big\{ d(s(\ell-1)+\ell)-1, \quad d((\ell-1)e-s)-1,\ldots,d(2e-s)-1,d(e-s)-1 \Big\}
    \label{eq:psi}
  \end{multline}
  for $0 \leq s < e$.
  We consider the defining polynomial
  \[
    Q(\CA,\mu) = x_1^{dm+1} \cdots x_\ell^{dm+1} \prod_{1\le i<j \le \ell} (x_i^{de}-x_j^{de})^{2m+1}.
  \]
  This is the case $m_1=m_2=\ldots=m_\ell=dm+1$ in \Cref{thm:main2strong}\eqref{eq:main2strong2}.
  We have $a=(\ell-1)de, q=b$ with uniquely written $m = be+s$ for $0\leq s < e$, and $m'=\ell dm+\ell$.
  Consequently, using~\eqref{eq:constant_h} and~\eqref{eq:psi}, $(\CA,m\omega+1)$ is free and 
  \begin{align*}
    \exp(\CA,m\omega+1) &= \Big\{ (m-b)(\ell-1)de+\ell dm+1, \\ 
                &\hspace*{25pt} (m(\ell-1)+b+1)de + 1, \ldots, (m(\ell-1)+b+(\ell-1))de + 1 \Big\} \\
                &= \Big\{ (m+1)((\ell-1)de+d) - d((e-s-1)(\ell-1)+\ell) + 1, \\
                &\hspace*{25pt} (m+1)((\ell-1)de+d) - d((\ell-1)e-(e-1-s)) + 1, \ldots,\\
                &\hspace*{35pt} (m+1)((\ell-1)de+d) - d(e-(e-1-s)) + 1 \Big\}\\  
                &= \big\{ (m+1)h -n_1(\Psi^{e-1-s}(V^*)^*),\ldots,(m+1)h -n_\ell(\Psi^{e-1-s}(V^*)^*) \big\}\\
                &= \big\{ (m+1)h -n_1(\Psi^{-m-1}(V^*)^*),\ldots,(m+1)h -n_\ell(\Psi^{-m-1}(V^*)^*) \big\}\\
                &= \big\{ mh+n_1(V^*),\ldots,mh+n_\ell(V^*) \big\},
  \end{align*}
  as claimed.
\end{proof}

In the remainder of this section, we prove the two parts of \Cref{thm:main2strong} separately.

\subsection{Proof of \Cref{thm:main2strong}\eqref{eq:main2strong1}}
\label{sec:thmmain2strong1}

We begin with the situation in rank~$2$ and set $S = \BBC[x,y]$ in this case.

\begin{lemma}
\label{prep:exp_gr1l_even}
  Let $r\ge 2$, $m \ge 1$ and $k \ge 0$.
  Define $(\CA,\mu)$ by
  \[
    Q(\CA,\mu) = x^{kr+1}y^{kr+1} (x^r-y^r)^{2m}.
  \]
  Then $(\CA,\mu)$ is free with
  \[
    \exp(\CA,\mu) = \big\{(m+k)r+1,(m+k)r+1 \big\}.
  \]
  Moreover, there are homogeneous polynomials $q_1, q_2 \in \BBZ[x,y]$ of degrees~$m$ and~$m-1$, respectively, such that
  \begin{enumerate}[(i)]
    \item all coefficients of $q_1,q_2$ are non-zero, and
    \item the homogeneous derivations
      \begin{align*}
        \theta_1 &\eqdef x^{kr+1} q_1(x^r,y^r) \partial_{x} + y^{kr+1} x^r q_2(x^r,y^r) \partial_{y}\\
        \theta_2 &\eqdef x^{kr+1} y^r q_2(y^r,x^r) \partial_{x} + y^{kr+1} q_1(y^r,x^r) \partial_{y}
      \end{align*}
      form a basis of $\FD(\CA,\mu)$.
  \end{enumerate}
\end{lemma}
\begin{proof}
  We aim to define $q_1$ and $q_2$ such that $\theta_1(x-y) \in (x-y)^{2m}S$.
  Let $q_1(x,y) = \sum_{i=0}^m a_i x^{i} y^{m-i}$ and $q_2(x,y) = \sum_{i=0}^{m-1} b_i x^{i} y^{m-1-i}$ for some $a_i, b_i \in \BBQ$.
  Then
  \[
  P(x,y) \eqdef \theta_1(x-y) = \sum_{i=0}^m a_i x^{(i+k)r+1} y^{(m-i)r} - 
                          \sum_{i=0}^{m-1} b_i x^{(i+1)r} y^{(m-1-i+k)r+1}.
  \]
  Since we require $P(x,y) \in (x-y)^{2m}S$, the coefficients $a_0,\ldots,a_m, b_0,\ldots,b_{m-1}$ form a solution of the following system of linear equations over~$\BBQ$
  \[
    \left(\left(\frac{d}{dx}\right)^j P\right)(x,x) = 0
  \]
  for $j=0,\ldots,2m-1$.

  The entries of the corresponding matrix are just given by the exponents of~$x$ in~$q_1,q_2$.
  Dividing the $j$-th equation by~$j!$, the entries of the respective equations become
  \[
    \binom{(i+k)r+1}{j} \quad \text{and} \quad -\binom{(i+1)r}{j}
  \]
  for~$a_i$ with $i=0,\ldots,m$, and, respectively, for $b_i$ with $i=0,\ldots,m-1$.
  We may avoid the minus sign by replacing~$b_i$ by $-b_i$.
  The homogeneous system has~$2m$ equations and $2m+1$ unknowns.
  Thus, we may choose a non-trivial solution with $a_i,b_i \in \BBZ$ for all~$i$.

  Now assume that one of the~$a_i$ or one of the~$b_i$ is zero, so that we may omit the corresponding summand in~$q_1$ or~$q_2$.
  This corresponds to deleting the coefficients of this monomial in the given system of equations.
  But then, the matrix for $j=0,\ldots,2m-1$ is a $(2m \times 2m)$-matrix, which is invertible thanks to the famous Gessel-Viennot lemma~\cite{gessel:binomialdeterminants}.
  In this case, there is only the trivial solution.
  This contradicts the fact that we have already obtained a non-trivial solution in the previous paragraph.
  Hence, none of the $a_i$ or $b_i$ are zero. 

  Next, we check that~$\theta_1 \in \FD(\CA,\mu)$. 
  By construction, $\theta_1(x) \in x^{kr+1}S$, $\theta_1(y) \in y^{kr+1}S$ and $\theta_1(x-y) \in (x-y)^{2m}S$.
  Then, for~$\zeta$ an $r$-th root of unity, we have
  \begin{align*}
  \theta_1(x-\zeta y) &= x^{kr+1}q_1(x^r,y^r) - \zeta y^{kr+1}x^r q_2(x^r,y^r)\\
                      &= x^{kr+1}q_1(x^r,(\zeta y)^r) - (\zeta y)^{kr+1}x^r q_2(x^r,(\zeta y)^r)\\
                      &= P(x,\zeta y) \in (x-\zeta y)^{2m}S.
  \end{align*}
  Hence $\theta_1\in \FD(\CA,\mu)$.
  Likewise, we also get that $\theta_2 \in \FD(\CA,\mu)$.
  Observe that
  \[
  \det M(\theta_1,\theta_2) = x^{kr+1}y^{kr+1} q_1(x^r,y^r)q_1(y^r,x^r) - x^{(k+1)r+1}y^{(k+1)r+1} q_2(x^r,y^r)q_2(y^r,x^r)
  \]
  is non-zero of degree $|\mu|$.
  (The first part is only divisible by $x^{kr+1}$ and the second part is divisible by $x^{(k+1)r+1}$.)
  Thus $\theta_1$ and $\theta_2$ are independent over $S$. 
  Since $\theta_1$ and $\theta_2$ are homogeneous and $\pdeg \theta_1+ \pdeg\theta_2 = |\mu|$, 
  it follows from \Cref{thm:zieglersaito}\eqref{eq:zieglersaito3} that 
  $\{\theta_1,\theta_2\}$ is a basis of $\FD(\CA,\mu)$).
\end{proof}

\begin{corollary}
\label{even_cor}
  Let $r\ge 2$, $m \ge 1$, $k \ge 0$ and $m_1,m_2 \ge 0$ such that $(k-1)r+1 \le m_1,m_2 \le kr+1$.
  Define $(\CA,\mu)$ by
  \[
    Q(\CA,\mu) = x^{m_1}y^{m_2} (x^r-y^r)^{2m}.
  \]
  Then $(\CA,\mu)$ is free with
  \[
    \exp(\CA,\mu) = \{rm+m_1,rm+m_2\}.
  \]
\end{corollary}

\begin{proof}
  A basis $\{\tilde\theta_1,\tilde\theta_2\}$ of $\FD(\CA,\mu)$ is given by
  \begin{equation*}
    \tilde\theta_1 = \begin{cases}
                       \theta_1 / x   & \text{ if } k=0, \\
                       \theta_1 / x^r & \text{ if } k>0,
                     \end{cases}
    \quad\text{and}\quad
    \tilde\theta_2 = \begin{cases}
                       \theta_2 / y   & \text{ if } k=0, \\
                       \theta_2 / y^r & \text{ if } k>0,
                     \end{cases}
  \end{equation*}
  where $\theta_1$ and $\theta_2$ are given as in \Cref{prep:exp_gr1l_even}.
\end{proof}

We next use the rank~$2$ considerations to prove the general rank~$\ell$ case.

\begin{theorem}
\label{even_theorem}
  Let $\ell,r \ge 2$, $k \ge 0$, $m_1,\ldots,m_\ell \ge 0$ and $(k-1)r+1 \le m_1,\ldots,m_\ell \le kr+1$.
  Define $(\CA,\mu)$ by
  \[
    Q(\CA,\mu) = x_1^{m_1} \cdots x_\ell^{m_\ell} \prod_{1 \le i<j \le \ell} (x_i^r - x_j^r)^{2m}.
  \]
  Then $\FD(\CA,\mu)$ is free with
  \[
    \exp(\CA,\mu) = \{c+m_1,\ldots,c+m_\ell\},
  \]
  where $c = (\ell-1)mr$.
\end{theorem}

\begin{proof}
  We argue by induction on~$\ell$.
  Thanks to \Cref{even_cor}, the theorem holds for $\ell=2$.
  Now, suppose $\ell > 2$.
  The proof in this case follows from a further induction on $\sum m_i$. 
  Thanks to \Cref{thm:main1}, the statement of the theorem holds for $m_1 = \ldots = m_\ell = 0$.
  Now let $\sum m_i \neq 0$.
  Without loss, we may assume that~$m_\ell > 0$ is maximal among the~$m_i$.
  We aim to apply \Cref{thm:add-del-multi} with respect to the hyperplane $H_\ell = \ker{x_\ell}$.
  If $m_1 = \cdots = m_\ell = (k-1)r+1$, then, in order to being able to apply the induction hypothesis, requiring the lower bounds on the~$m_i$, we replace~$k$ by~$k-1$.
  Observe that this replacement is valid as, crucially, the exponents do not depend on~$k$.

  The defining polynomial of the deletion with respect to $H_\ell$ is given by
  \[
    Q(\CA',\mu') = x_1^{m_1} \cdots x_{\ell-1}^{m_{\ell-1}}x_\ell^{m_\ell-1} \prod_{1 \le i<j \le \ell} (x_i^r - x_j^r)^{2m}.
  \]

  Now, by induction on $\sum m_i$, $(\CA',\mu')$ is free with exponents
  \[
    \exp(\CA',\mu') = \{c+m_1,\ldots,c+m_{\ell-1},c+m_\ell-1\}.
  \]
  The Euler multiplicity $\mu^*$ on $\CA^{H_\ell}$ is given by
  \[
    Q(\CA^{H_\ell},\mu^*) = x_1^{m_1 + rm} \cdots x_{\ell-1}^{m_{\ell-1} + rm } \prod_{1 \le i<j \le \ell-1} (x_i^r - x_j^r)^{2m}.
  \]
  This can be seen as follows.
  For a hyperplane $H_{ij} = \ker(x_i-x_j)$ ($i,j \not=\ell$) the localization is of size $\vert \CA_{H_\ell \cap H_{ij}}\vert = 2$, hence the Euler multiplicity is~$2m$, by \Cref{prop:Euler}.
  For a hyperplane $H_i = \ker x_i$ ($i \neq \ell$) the localization is given by
  \[
    x_i^{m_i}x_\ell^{m_\ell} (x_i^r-x_\ell^r)^{2m}
  \]
  with exponents $\{rm+m_i, rm+m_\ell\}$, by \Cref{even_cor}. By decreasing $m_\ell$, the second exponent changes, again according to \Cref{even_cor}.
  Hence the Euler multiplicity is $rm+m_i$.

  Now, by induction on~$\ell$, we know that $(\CA^{H_\ell},\mu^*)$ is free, and we compute the exponents as follows.
  The corresponding constant $c^*$ from the statement of the theorem is $c^* = ((\ell-1)-1)mr = c - mr$.
  Hence,
  \[
    \exp(\CA^{H_\ell},\mu^*) = \{c^*+m_1+mr,\ldots,c^*+m_{\ell-1}+mr\} = \{c + m_1, \ldots, c+m_{\ell-1}\}.
  \]
  The theorem now follows by \Cref{thm:add-del-multi}.
\end{proof}

Note that \Cref{thm:main2strong}\eqref{eq:main2strong1} follows from \Cref{even_theorem}.

\subsection{Proof of \Cref{thm:main2strong}\eqref{eq:main2strong2}}
\label{sec:thmmain2strong2}

The derivations are constructed in a similar way as in the previous case.
Hence we construct a polynomial whose coefficients are the solution of a system of linear equations which depend on several indeterminates.
The key observation in the previous case was the regularity of a matrix whose entries consist of certain binomial coefficients.
It turns out that in the present case the entries of the matrix consist of differences of certain binomial coefficients.
The application of the following technical lemma in the present situation was communicated to us by Christian Krattenthaler.

\begin{lemma}[{\cite[Lem.~7]{krattenthaler:determinantcalculus}}]
\label{prep:detlemma}
  Let $X_1,X_2,\ldots, X_m$, $A_2,A_3,\ldots,A_m,C$ be indeterminates, and let $p_0,p_1,\ldots,p_{m-1}$ 
  be polynomials in a single variable such that $\deg p_j \le 2j$ and $p_j(X) = p_j(C-X)$ for $j=0,1,\ldots,m-1$.
  Then,
    \begin{multline*}
  \det_{1\le i,j\le m} \left(\prod_{k=j+1}^m \Big((X_i-A_k)(X_i-A_k-C)\Big) \cdot p_{j-1}(X_i)\right)\\
    = \prod_{1\le i <j \le m} (X_j-X_i)(C-X_i-X_j) \prod_{i=1}^m p_{i-1}(-A_i).
  \end{multline*}
\end{lemma}
Comparing the coefficient of $A_m^{2m-2}A_{m-1}^{2m-4}\cdots A_1^0$ in the identity in the previous lemma, we obtain
\begin{equation}
\det_{1\le i,j \le m} \big(p_{j-1}(X_i)\big) = \prod_{1\le i <j \le m} (X_i-X_j)(C-X_i-X_j) \prod_{i=1}^m q_{i-1}
\label{eq:krattenthalerspecial}
\end{equation}
where~$q_j$ is the leading coefficient of $p_j(X)$.

Utilizing~\eqref{eq:krattenthalerspecial}, we obain the following consequence.

\begin{corollary}
\label{detcorollary}
  Let $A,B,m \in \BBN$ such that $A-B \not\equiv 0 \mod r$, and set $C=\frac{B-A}{r}$, $X_i = i-1$,
  and
  \[
    p_j(X) = \frac{1}{A-B+2rX} \Bigg(\binom{A+rX}{2j+1} - \binom{B-rX}{2j+1} \Bigg) \in \BBQ[X].
  \]
  Then the left-hand side of~\eqref{eq:krattenthalerspecial} specializes to
  \[ 
  \det_{0 \le i,j \le m-1} \Bigg(\binom{A+ri}{2j+1} - \binom{B-ri}{2j+1}\Bigg),
  \]
  which is not identically zero.
\end{corollary}

We use this corollary repeatedly in the subsequent lemma.

\begin{lemma}
\label{prep:exp_gr1l_odd}
  Let $m \ge 1$ and $k \ge 0$.
  Define $(\CA,\mu)$ by
  \[
    Q(\CA,\mu) = x^{kr+1}y^{kr+1} (x^r-y^r)^{2m+1}.
  \]
  Then $(\CA,\mu)$ is free with
  \[
    \exp(\CA,\mu) = \big\{(m+k)r+1,(m+k+1)r+1\big\}.
  \]
  Moreover, there are polynomials $q_1,q_2 \in \BBZ[x,y]$ of degree~$m$ such 
  that
  \begin{enumerate}[(i)]
    \item the coefficients of $x^m$ and $y^m$ in $q_1,q_2$ are non-zero and 
    \item the homogeneous derivations
      \begin{align*}
      \theta_1 &\eqdef x^{kr+1} q_1(x^r,y^r) \partial_{x} + y^{kr+1} q_1(y^r,x^r) \partial_{y}\\
      \theta_2 &\eqdef x^{kr+1} y^r q_2(x^r,y^r) \partial_{x} + y^{kr+1} x^r q_2(y^r,x^r) \partial_{y}\\
      \end{align*}
      form a basis of $\FD(\CA,\mu)$.
  \end{enumerate}
\end{lemma}
\begin{proof}
  We aim to define~$q_1$ such that $\theta_1(x-y) \in (x-y)^{2m+1}S$.
  Let~$q_1 = \sum_{i=0}^m a_i x^i y^{m-i}$ with~$a_i \in \BBQ$.
  We require that
  \[
    P(x,y) := \theta_1(x-y) = \sum_{i=0}^m a_i \left( x^{(i+k)r+1}y^{(m-i)r} - x^{(m-i)r}y^{(i+k)r+1}\right) \in (x-y)^{2m+1}S.
  \]
  Hence, the coefficients $a_0,\ldots,a_m$ form a solution of the following system of linear equations over~$\BBQ$
  \begin{equation}
    \label{eq:matrixequation1}
    \frac{1}{j!}\left(\left(\frac{d}{dx}\right)^{j} P\right)(x,x) = 0    
  \end{equation}
  for $j=0,\ldots,2m$.
  Since $P(x,y) = - P(y,x)$, the identity~\eqref{eq:matrixequation1} holds for a given even~$j$, provided it holds for all~$j'$ with $0 \leq j' < j$.
  (In particular, it holds for $j=0$.)

  Because we have $m+1$ variables and~$m$ equations, the system has a non-trivial solution.
  We may choose one such non-trivial solution with coefficients in~$\BBZ$.

  Suppose $a_m = 0$.
  Then we may remove the last column of the matrix in~\eqref{eq:matrixequation1}.
  The determinant of this matrix equals
  \begin{equation*}
    \det_{0\le i,j \le m-1} \Bigg(\binom{(i+k)r+1}{2j+1} - \binom{(m-i)r}{2j+1}\Bigg)
  \end{equation*}
  which is not identically zero, thanks to \Cref{detcorollary} for the parameters $A= kr+1$ and $B=mr$. 
  Hence,~\eqref{eq:matrixequation1} only admits the trivial solution, contradicting the above choice of a non-trivial solution.

  \medskip

  Suppose $a_0 = 0$.
  Then we may remove the first column of the matrix in~\eqref{eq:matrixequation1}.
  Its determinant equals, after substituting~$i$ by~$i+1$, the determinant
  \begin{equation*}
    \det_{0\le i,j \le m-1} \Bigg(\binom{(i+k+1)r+1}{2j+1} - \binom{(m-i-1)r}{2j+1}\Bigg)
  \end{equation*}
  which is not identically zero, again thanks to \Cref{detcorollary} for the parameters $A= (k+1)r+1$ and $B=(m-1)r$.
  Hence, we obtain an analogous contradiction as in the previous case.

  \medskip

  Next, we check that $\theta_1 \in \FD(\CA,\mu)$.
  By construction, $\theta_1(x) \in x^{kr+1}S$, $\theta_1(y) \in y^{kr+1}S$ and $\theta_1(x-y) \in (x-y)^{2m+1}S$.
  Then, for~$\zeta$ an $r$-th root of unity, we have
  \begin{align*}
  \theta_1(x-\zeta y) &= x^{kr+1} q_1(x^r,y^r) - \zeta y^{kr+1}q_1(y^r,x^r)\\
                      &= x^{kr+1} q_1(x^r,(\zeta y)^r) - (\zeta y)^{kr+1}q_1((\zeta y)^r,x^r)\\
                      &= P(x,\zeta y) \in (x-\zeta y)^{2m+1} S.
  \end{align*}

  We aim to define~$q_2$ such that $\theta_2(x-y) \in (x-y)^{2m+1}S$.
  Let $q_2 = \sum_{i=0}^m b_i x^i y^{m-i}$ with $a_i \in \BBQ$.
  We require that
  \[
    \tilde P(x,y) := \theta_2(x-y) = \sum_{i=0}^m b_i \left( x^{(i+k)r+1}y^{(m-i+1)r} - x^{(m-i+1)r}y^{(k+i)r+1}\right) \in (x-y)^{2m+1}S.
  \]
  Hence the coefficients of~$q_2$ are the solutions of the system of equations given by
  \begin{equation}
    \label{eq:matrixequation2}
    \frac{1}{j!}\left(\left(\frac{d}{dx}\right)^{j} \tilde P\right)(x,x) = 0
  \end{equation}
  for $j=0,\ldots,2m$.
  As above, since $\tilde P (x,y) = - \tilde P (y,x)$, we again observe that the equation holds for a given even~$j$, provided it holds for all~$j'$ with $0 \leq j' < j$.

  Because we have $m+1$ variables and~$m$ equations, the system has a non-trivial solution.
  We may choose one such non-trivial solution with coefficients in~$\BBZ$.

  Suppose $b_m = 0$.
  Then we may remove the last column of the matrix in~\eqref{eq:matrixequation2}.
  The determinant of this matrix equals
  \[
    \det_{0\le i,j \le m-1} \Bigg(\binom{(i+k)r+1}{2j+1} - \binom{(m-i+1)r}{2j+1}\Bigg)
  \]
  which is not identically zero, thanks to \Cref{detcorollary} for the parameters $A= kr+1$ and $B=(m+1)r$.
  Hence,~\eqref{eq:matrixequation2} only admits the trivial solution, contradicting the above choice of a non-trivial solution.

  Suppose $b_0 = 0$.
  Then we may remove the first column of the matrix in~\eqref{eq:matrixequation2}.
  Its determinant equals, after substituting~$i$ by~$i+1$, the determinant
  \[
    \det_{0\le i,j \le m-1} \Bigg(\binom{(i+k+1)r+1}{2j+1} - \binom{(m-i)r}{2j+1}\Bigg)
  \]
  which is not identically zero, again thanks to \Cref{detcorollary} for the parameters $A= (k+1)r+1$ and $B=mr$.
  Hence, we obtain an analogous contradiction as in the previous case.

  \medskip

  Finally, we check that $\theta_2 \in \FD(\CA,\mu)$.
  By construction, $\theta_2(x) \in x^{kr+1}S, \theta_2(y) \in y^{kr+1}S$ and $\theta_2(x-y) \in (x-y)^{2m+1}S$.
  Then, for~$\zeta$ an $r$-th root of unity, we have
  \begin{align*}
  \theta_2(x-\zeta y) &= x^{kr+1} y^r q_2(x^r,y^r) - \zeta y^{kr+1}x^rq_2(y^r,x^r)\\
                      &= x^{kr+1} (\zeta y)^rq_2(x^r,(\zeta y)^r) - (\zeta y)^{kr+1} x^r q_2((\zeta y)^r,x^r)\\
                      &= \tilde P (x,\zeta y) \in (x-\zeta y)^{2m+1} S.\\
  \end{align*} 
  Hence $\theta_1,\theta_2 \in \FD(\CA,\mu)$.
  Observe that
  \[
    \det M(\theta_1,\theta_2) = x^{(k+1)r+1} y^{kr+1} q_1(x^r,y^r)q_2(y^r,x^r)- y^{(k+1)r+1} x^{kr+1} q_1(y^r,x^r)q_2(x^r,y^r).
  \]
  This determinant has degree $\vert\mu\vert$.
  Since~$q_1$ and~$q_2$ are not divisible by~$x$ and~$y$, respectively, the determinant is non-zero. 
  Thus $\theta_1$ and $\theta_2$ are independent over $S$. 
  Since $\theta_1$ and $\theta_2$ are homogeneous and $\pdeg \theta_1+ \pdeg\theta_2 = |\mu|$, 
  it follows from \Cref{thm:zieglersaito}\eqref{eq:zieglersaito3} that 
  $\{\theta_1,\theta_2\}$ is a basis of $\FD(\CA,\mu)$).
\end{proof}

\begin{corollary}
\label{odd_cor}
  Let $m \ge 1$ and $k \ge 0$.
  Let $0 \le \tilde m_1, \tilde m_2 \le r$ and set $m_i \eqdef (k-1)r+1 + \tilde m_i$ for $i = 1,2$ such that $m_1,m_2 \ge 0$.
  Define $(\CA,\mu)$ by
  \[
    Q(\CA,\mu) = x^{m_1}y^{m_2} (x^r-y^r)^{2m+1}.
  \]
  Let $c \eqdef (k-1+m)r +1$.
    Then $(\CA,\mu)$ is free with
  \[
    \exp(\CA,\mu) = \{c + \tilde m_1 + \tilde m_2,\ c+r \}.
  \]
\end{corollary}

\begin{proof}
  A basis $\{\tilde\theta_1,\tilde\theta_2\}$ of $\FD(\CA,\mu)$ is given by $\tilde\theta_1 = \theta_1$ and
  \[
    \tilde\theta_2 = \begin{cases}
                       \theta_2 / xy     & \text{ if } k=0, \\
                       \theta_2 / x^ry^r & \text{ if } k>0,
                     \end{cases}
  \]
  where $\theta_1$ and $\theta_2$ are given as in \Cref{prep:exp_gr1l_odd}.
\end{proof}

\begin{theorem}
\label{odd_theorem}
  Let $k,m \ge 0$ and $0 \le \tilde m_1,\ldots,\tilde m_\ell \le r$.
  Set $m_i \eqdef (k-1)r+1+\tilde m_i$ for $i \in \{1,\ldots,\ell\}$ such that $m_i \ge 0$, and set $c \eqdef ((\ell-1)m + k-1)r +1$.

  Define $(\CA,\mu)$ by
  \[
    Q(\CA,\mu) = x_1^{m_1}\cdots x_\ell^{m_\ell} \prod_{1 \le i < j \le \ell} (x_i^r-x_j^r)^{2m+1}.
  \]
  Then $(\CA,\mu)$ is free with 
  \[
    \exp(\CA,\mu) = \Big\{ c+\sum_{i=1}^\ell \tilde m_i,\ c+r,\ c+2r,\ \ldots,\ c+(\ell-1)r \Big\}.
  \]
\end{theorem}

\begin{proof}
  We argue by induction on $\ell$.
  By \Cref{odd_cor}, the theorem holds for $\ell = 2$.

  Suppose $\ell > 2$.
  The proof in this case follows from a further induction on $\sum m_i$.
  Thanks to \Cref{thm:main1}, the theorem holds for $m_1 = \cdots = m_\ell = 0$.
  Now let $\sum m_i \neq 0$.
  Without loss, we may assume that $m_\ell > 0$ is maximal among the~$m_i$.
  We aim to apply \Cref{thm:add-del-multi} with respect to the hyperplane $H_\ell = \ker{x_\ell}$.
  If $m_1 = \cdots = m_\ell = (k-1)r+1$, then, in order to being able to apply the induction hypothesis, requiring the lower bounds on the~$m_i$, we replace~$k$ by~$k-1$, and, simultaneously, replace $\tilde m_i = 0$ by $\tilde m_i = r$ for all~$i$.
  Observe that this replacement is valid as it does not change the arrangement and, crucially, the exponents also coincide in both cases.

  \medskip

  The defining polynomial of the deletion with respect to $H_\ell$ is given by
  \[
    Q(\CA',\mu') = x_1^{m_1} \cdots x_{\ell-1}^{m_{\ell-1}}x_\ell^{m_\ell-1} \prod_{1 \le i<j \le \ell} (x_i^r - x_j^r)^{2m+1}.
  \]

  Now, by induction on $\sum m_i$, the deletion $(\CA',\mu')$ is free with exponents
  \[
  \exp(\CA',\mu') = \Big\{ c-1+\sum_{i=1}^\ell \tilde m_i, \ c+r, \ c+2r, \ \ldots, \ c+(\ell-1)r \Big\}.
  \]

  The Euler multiplicity $\mu^*$ on $\CA^{H_\ell}$ is given by
  \[
    Q(\CA^{H_\ell},\mu^*) = x_1^{(k+m)r+1} \cdots x_{\ell-1}^{(k+m)r+1} \prod_{1 \le i<j \le \ell-1} (x_i^r - x_j^r)^{2m+1}.
  \]
  This can be seen as follows.

  For a hyperplane $H_{ij} = \ker(x_i-x_j)$ ($i,j \not=\ell$) the localization is of size $\vert \CA_{H_\ell \cap H_{ij}}\vert = 2$, hence the Euler multiplicity is~$2m+1$ by \Cref{prop:Euler}.
  For a hyperplane $H_i = \ker x_i$ ($i \neq \ell$) the localization is given by
  \[
    x_i^{m_i}x_\ell^{m_\ell} (x_i^r-x_\ell^r)^{2m+1}
  \]
  with exponents $(\tilde c+\tilde m_i + \tilde m_\ell, (k+m)r+1)$, by \Cref{odd_cor}, and by decreasing $m_\ell$, the first exponent changes, again according to \Cref{odd_cor}.
  Hence the Euler multiplicity is $(k+m)r+1$.

  Now, by induction on~$\ell$, we know that $(\CA^{H_\ell},\mu^*)$ is free, and we compute the exponents as follows.
  The corresponding constant $c^*$ from the statement of the theorem is $c^* = ((l-2)m + k + m)r +1 = ((l-1)m+k-1)r + r +1 = c+r$.
  Hence,
  \[
    \exp(\CA^{H_\ell},\mu^*) = \{ c^*, \ c^*+r, \ \ldots, \ c^*+(\ell-2)r\} = \{ c+r, \ c+2r, \ \ldots, \ c+(\ell-1)r \}.
  \]
  The theorem now follows by \Cref{thm:add-del-multi}.
\end{proof}

Note that \Cref{thm:main2strong}\eqref{eq:main2strong2} follows from \Cref{odd_theorem}.

\bigskip

\section*{Acknowledgments}

We thank Masahiko Yoshinaga for explaining the details of the universality in \Cref{prop:univerality} in the case of Coxeter arrangements and Christian Krattenthaler for pointing out how to apply~\cite[Lem.~7]{krattenthaler:determinantcalculus} to our situation in \Cref{detcorollary}.
C.S. also thanks Anne Shepler for detailed discussions on the topic of this paper.

\medskip

We acknowledge support from the DFG priority program SPP1489 ``Algorithmic and Experimental Methods in Algebra, Geometry, and Number Theory''.
Part of the research for this paper was carried out while three of us
(T.H., G.R.~and C.S.) were staying at the Mathematical Research Institute
Oberwolfach supported by the ``Research in Pairs'' program.
T.M.~was supported in part by JSPS KAKENHI Grant Numbers 25800082, 17K05335.
C.S.~was was supported by the DFG grants STU 563/2 ``Coxeter-Catalan combinatorics'' and STU 563/4-1 ``Noncrossing phenomena in Algebra and Geometry''.


\bigskip

\bibliographystyle{amsalpha}

\newcommand{\etalchar}[1]{$^{#1}$}
\providecommand{\bysame}{\leavevmode\hbox to3em{\hrulefill}\thinspace}
\providecommand{\MR}{\relax\ifhmode\unskip\space\fi MR }
\providecommand{\MRhref}[2]{%
  \href{http://www.ams.org/mathscinet-getitem?mr=#1}{#2} }
\providecommand{\href}[2]{#2}


\end{document}